\documentclass[11pt]{amsart}

\usepackage[top=1in, bottom=1in, left=1in, right=1in]{geometry}
\usepackage[dvipsnames]{xcolor}
\usepackage[colorlinks=true, pdfstartview=FitV, linkcolor=RoyalBlue,citecolor=ForestGreen, urlcolor=blue]{hyperref}
\usepackage{xfrac}

\usepackage[english]{babel}
\usepackage{graphicx} 
\usepackage{amsthm} 
\usepackage{amsmath}
\usepackage{amssymb}
\usepackage{amsthm} 
\usepackage{mathrsfs} 
\usepackage{bbm} 
\usepackage{enumerate} 
\usepackage[shortlabels]{enumitem}
\usepackage{xcolor}
\usepackage{nicefrac}
\usepackage{subfigure}

\usepackage{hyperref}

\newcommand{\N}{\mathbb{N}}
\newcommand{\Z}{\mathbb{Z}}

\newcommand{\R}{\mathbb{R}}
\newcommand{\C}{\mathbb{C}}
\newcommand{\T}{\mathbb{T}}

\newcommand{\dif}{\,\mathrm{d}}

\newtheorem{lemma}{Lemma}[section]
\newtheorem{prop}[lemma]{Proposition}
\newtheorem{thm}[lemma]{Theorem}
\newtheorem{defi}[lemma]{Definition}
\newtheorem{cor}[lemma]{Corollary}
\newtheorem{Rem}[lemma]{Remark}

\title{A proof of Vishik's nonuniqueness Theorem for the forced 2D Euler equation}
\author{\'Angel Castro, Daniel Faraco, Francisco Mengual, and Marcos Solera}
\date{\today}

\begin{document}

\begin{abstract}
We give a simpler proof of Vishik's nonuniqueness Theorem for the forced 2D Euler equation in the vorticity class $L^1\cap L^p$ with $2<p<\infty$. The main simplification is an alternative construction of a smooth and compactly supported unstable vortex, which is split into two steps: Firstly, we construct a piecewise constant unstable vortex, and secondly, we find a regularization through a fixed point argument. This simpler structure of the unstable vortex yields a simplification of the other parts of Vishik's proof.
\end{abstract}
\maketitle
\section{Introduction}

We consider the Euler equation in vorticity form
\begin{equation}\label{eq:Euler}
\partial_t\omega + v\cdot\nabla\omega=f,
\end{equation}
posed on $\R_+\times\R^2$,
for some fixed external force $f$ and initial datum $\omega^\circ$
\begin{equation}\label{eq:Euler:2}
\omega|_{t=0}
=\omega^\circ.
\end{equation}
The velocity $v$ is recovered from the vorticity $\omega$ through the Biot-Savart law
\begin{equation}\label{eq:BiotSavart}
v=\nabla^\perp\Delta^{-1}\omega.
\end{equation}
The aim of this work is to simplify the proof of Vishik's nonuniqueness Theorem \cite{Vishikpp1,Vishikpp2}:

\begin{thm}\label{thm:Vishik}
For any $2<p<\infty$ there exists a force
\begin{equation}\label{forceintegrability}
f\in L_t^1(L^1\cap L^p)
\quad\text{with}\quad
g\in L_t^1L^2,
\end{equation}
where $g=\nabla^\perp\Delta^{-1} f$, 
with the property that there are two different solutions
\begin{equation}\label{vorticityintegrability}
\omega_0,\omega_1\in L_t^\infty (L^1\cap L^p)
\quad\text{with}\quad
v_0,v_1\in L_t^\infty L^2,
\end{equation}
where $v_j=\nabla^\perp\Delta^{-1}\omega_j$, to the Euler equation starting from  $\omega^\circ=0$.
\end{thm}

Theorem \ref{thm:Vishik} shows sharpness of Yudovich's uniqueness Theorem for the forced Euler equation. More precisely, Yudovich's theory implies that if $p=\infty$ in \eqref{forceintegrability}\eqref{vorticityintegrability}, then necessarily $\omega_0=\omega_1$.

In fact, in the construction of Theorem~\ref{thm:Vishik}, both the force $(f,g)$ and the main solution $(\omega_0,v_0)$ belong to $C_c^\infty$ for positive times. Moreover, for any given $k\in\N$, it is not difficult to see that the second solution $(\omega_1,v_1)$ can be upgraded to be in $H^{k}\times H^{k+1}$ for positive times. These estimates inevitably deteriorate as $t\to 0$. As this information is not relevant in Yudovich's proof of uniqueness, we have chosen not to include it in  the statement of Theorem \ref{thm:Vishik}.

In a nutshell, Vishik's proof is split into three steps: (1) Construction of an unstable vortex, which is then shown to be  (2) Self-similarly unstable, and (3) Nonlinearly unstable as well.

Step (1) is carried out in \cite{Vishikpp1,Vishikpp2} by an intricate modification of the power-law vortex $\omega^\circ(x)=\beta|x|^{-\alpha}$. As noticed in \cite{ABCDGMKpp}, by decoupling the parameters governing the self-similar scaling from the
decay rate, Vishik's argument leads to nonuniqueness for the zero initial data $\omega^\circ=0$, which showcases the primary role played by the forcing $f$.

The observation in \cite{ABCDGMKpp} that decay does not play a significant role in Vishik's Theorem
suggested us the possibility of obtaining  the proof by considering compactly supported vortices. Even more, inspired by simple examples for shear flows (see e.g.~\cite[Chapter 4]{DrazinReid04}) and vortex patches, we initially aimed to construct a piecewise constant unstable vortex.
Remarkably, Step (1) for piecewise constant vortices reduces to an elementary computation. In order to upgrade the piecewise constant vortex to a smooth vortex, we need to apply a fixed point argument in suitable rescaled variables. This step is inspired in \cite{CastroLear23} but needs of new ideas to be carried out here.

With  the new unstable vortex available, the proofs of Steps (2) and (3)  from Vishik or \cite{ABCDGMKpp} would work verbatim. However, since the new vortex is structurally simple, both Steps (2) and (3) admit also several simplifications leading to a straightforward self-contained proof of the entire theorem, which is presented in this work. In any case, the last two steps follow the presentation from book \cite{ABCDGMKpp} by Albritton, Bru\'e, Colombo, De Lellis, Giri, Janisch, and Kwon, simplifying the arguments whenever possible. 

As highlighted in Vishik's works \cite{Vishikpp1,Vishikpp2} and the book \cite{ABCDGMKpp}, the rigorous construction of the unstable vortex is a key intermediate step, which is of independent physical and mathematical interest. 
Indeed,  Albritton, Bru\'e and Colombo  \cite{ABC22} proved the nonuniqueness of Leray solutions for the forced 3D Navier-Stokes equation by carefully adapting Vishik's unstable vortex into the cross section of an axisymmetric vortex ring. We remark that, although our vortex differs from Vishik's one, it fits into the requirements of the paper \cite{ABC22} leading to  an alternative example of nonuniqueness for the forced 3D Navier-Stokes equation. 

We finally remark that our unstable vortices appear robust enough to be found in more general settings. Indeed, our investigation started because it was not obvious how Vishik's strategy  could be modified to prove nonuniqueness in the more singular (and nonlocal) context of generalized SQG. In this case, the Biot-Savart law \eqref{eq:BiotSavart} is replaced by $v=-\nabla^{\perp}(-\Delta)^{\nicefrac{\alpha}{2}-1} \omega$.
The construction and regularization of piecewise constant unstable vortices presented in this paper are however rather flexible and have the potential to work even for $\alpha$ all the way up to $1$.  
Tackling the SQG case ($\alpha=1$) poses even more challenges, notably because $v$ is not obtained through a smoothing operator. Both cases, generalized SQG and SQG will be the matter of a forthcoming paper.

\subsection{Brief background} 

The global existence and uniqueness of classical solutions to the 2D Euler equation has been known for almost a century, starting with the works of Wolibner \cite{Wolibner33} and H\"older \cite{Holder33} in $C^{k,\gamma}$ spaces for $k=0,1,\ldots$ and $\gamma>0$. 
Local well-posedness was previously established by Lichtenstein \cite{Lichtenstein25} and Gunther \cite{Gunther27}.
The global well-posedness in 2D sharply contrasts with the 3D case due to vortex stretching, as demonstrated by Elgindi \cite{Elgindi21}, who showed finite-time singularities
with $\omega^\circ\in C^{\gamma}$ for some $0<\gamma\ll 1$ in the unforced case. 
Recently, C\'{o}rdoba, Mart\'{i}nez-Zoroa and Zheng \cite{CMZpp} simplified the proof of this blow-up with a different strategy.
Additionally, C\'{o}rdoba and Mart\'{i}nez-Zoroa \cite{CMpp} upgrade the regularity in the forced case with
$f\in L_t^\infty C^{\nicefrac{1}{2}-\varepsilon}$.
We refer to the work \cite{KiselevSverak14} by Kiselev and \v{S}ver\'{a}k for the related issue on small scale creation in two dimensions.

In the celebrated paper \cite{Yudovich63}, Yudovich extended the 2D global well-posedness into the vorticity class $L^1\cap L^\infty$. 
The existence of global solutions in $L^1\cap L^p$ for $1<p\leq\infty$ was proved by DiPerna and Majda \cite{DiPernaMajda87} at the late eighties.
Since then, a long-standing open question in the field is whether uniqueness fails for $p<\infty$. 

In the pioneering works \cite{JiaSverak14,JiaSverak15}, Jia and \v{S}ver\'{a}k introduced, within the context of the 3D Navier-Stokes equation, the idea that nonuniqueness could potentially arise from a self-similar instability. 
Given that uniqueness is known to be satisfied for small data (in suitable spaces, see e.g.~\cite{KochTataru01}), they conjectured that nonuniqueness could occur due to bifurcations within solutions of the form $\beta\bar{\omega}$, for some self-similar steady state. Specifically, they proved  in the celebrated paper \cite{JiaSverak14} the existence of self-similar solutions for large data. Later,  in \cite{JiaSverak15}, the authors speculated that some eigenvalues of the linearization in self-similar coordinates might cross the imaginary axis as $\beta$ increases, leading to the emergence of multiple solutions. While there are numerical evidences, due to Guillod and \v{S}ver\'{a}k \cite{GuillodSverak23}, supporting the existence of such solutions, the rigorous proof appears to be elusive, primarily because the Jia-\v{S}ver\'{a}k  self-similar  solutions are not explicit.

In the recent groundbreaking works \cite{Vishikpp1,Vishikpp2}, Vishik was able to construct a self-similarly unstable vortex for the 2D Euler equation, but at the cost of introducing a force. 
One of the innovative ideas in these works is Vishik's key observation that, by taking $\beta$ sufficiently large,  in order to find a self-similar unstable vortex it suffices to find an unstable vortex   
$\bar\omega$ in the original Eulerian coordinates. This  simplifies the analysis considerably. In the book \cite{ABCDGMKpp}, upon which Sections \ref{sec:selfsimilar} and \ref{sec:nonlinear} of the present work are based, the authors reviewed Vishik's nonuniqueness Theorem \ref{thm:Vishik}, providing more details, offering some simplifications and clarifying certain subtle points.

To the best of our knowledge, the question of nonuniqueness without forcing, both for the 2D Euler equation below the Yudovich class and for the 3D Navier-Stokes equation in the Leray class, remains unresolved to date.

In \cite{BressanMurray20,BressanShen21} Bressan, Murray, and Shen 
presented numerical evidences of the nonuniqueness for the unforced 2D Euler equation. 
Specifically, they introduced two different ways of regularizing a cleverly designed initial vorticity, leading to either one or two algebraic spirals. 
Recent research on these spirals by Garc\'ia and G\'omez-Serrano \cite{GGS22}, and by Shao, Wei and Zhang \cite{SWZ23}, based on the earlier work of Elling \cite{Elling16a}, could be relevant in the potential proof of nonuniqueness without force. 

The first examples of nonuniqueness for weak solutions can be attributed to Scheffer \cite{Scheffer93} and Shnirelman \cite{Shnirelman97}. Specifically, they constructed weak solutions $v\in L_{t,x}^2$ with compact support in time, to the Euler equation in velocity form.
In their seminal work \cite{DeLellisSzekelyhidi09}, De Lellis and Sz\'ekelyhidi introduced the convex integration method in hydrodynamics, constructing non-trivial Euler velocities within the energy space $L_t^\infty L^2$ for any space dimension $d\geq 2$. Over the last years, there has been a significant increase in research intensity focused on the method,
showcasing its remarkable robustness and flexibility.
As a pivotal landmark, this method allowed constructing Euler velocities in $C_{t,x}^{\nicefrac{1}{3}-\varepsilon}$ 
exhibiting nonuniqueness \cite{Isett18} by Isett, and dissipating the kinetic energy \cite{BDSV19} by Buckmaster, De Lellis, Sz\'ekelyhidi, and Vicol, thus solving the dissipative part of the 3D Onsager conjecture. See also the work \cite{NovackVicol23} by Novack and Vicol on an intermittent Onsager theorem, based on their joint work \cite{BMNV23} with Buckmaster and Masmoudi, as well as the recent work \cite{GNK23} on the $L^3$-based strong Onsager conjecture by Giri, Kwon and Novack. 
The conservative part of the Onsager conjecture in any dimension was partially proved by Eyink \cite{Eyink94} and later proven in full by Constantin, E, and Titi \cite{CET94}. See also the work \cite{CCFS08} by Cheskidov, Constantin, Friedlander, and Shvydkoy on critical regularity.
The dissipative part of the 2D Onsager conjecture was solved recently in \cite{GR23} by Giri and Radu by means of convex integration. 
We refer to the recent
work \cite{DeRosaParkpp} by De Rosa and Park for the related issue on anomalous dissipation in two dimensions.
In the case with force, Bulut, Huynh and Palasek proved in \cite{BKP23} that the regularity in 3D can be upgraded to $C^{\nicefrac{1}{2}-\varepsilon}$.
For other applications of convex integration to forced equations, we refer to the recent work \cite{DaiFriedlander23} by Dai and Friedlander,  and the references therein.

As already mentioned in \cite{ABCDGMKpp}, Theorem \ref{thm:Vishik} implies that, for any $\gamma<1$,
there exist two different Euler velocities $v_1,v_2\in L_t^\infty(L^2\cap C^\gamma)$ with force $g\in L_t^1(L^2\cap C^\gamma)$ starting from $v^\circ=0$.
To the best of our knowledge, the problem of nonuniqueness for the unforced Euler equation in the velocity class $L^2\cap C^\gamma$ remains open in the regime $\nicefrac{1}{3}\leq\gamma<1$.

Unfortunately, it seems not possible with the current convex integration techniques to construct
solutions, neither with $v\in C^{\gamma}$ for $\nicefrac{1}{3}\leq\gamma <1$, nor with $\omega\in L^p$ for $p\geq 1$ in two dimensions.
Despite this inconvenience, some partial results 
have been obtained in the last years. In \cite{Mpp}
the third author proved the existence, for any $2<p<\infty$, of initial data $v^\circ\in L^2$ with $\omega^\circ\in L^1\cap L^p$ for which there are infinitely many admissible solutions $v\in C_t L^2$. This result shows sharpness of Yudovich's proof of uniqueness, but with the drawback that the vorticity information is lost for positive times. In \cite{BrueColombo23} Bru\'e and Colombo constructed a Cauchy sequence $\omega_k$ in the Lorentz space $L^{1,\infty}$, whose velocities $v_k$ converge to an anomalous weak solution $v$. 
In \cite{BuckModenapp} Buck and Modena adapted the previous construction for the Hardy space $H^p$ for $\nicefrac{2}{3}<p<1$. This space is also weaker than $L^1$ but, in contrast to $L^{1,\infty}$, it already embeds into the space of distributions.

Finally, we emphasize that the convex integration method allowed for proving nonuniqueness of Navier-Stokes solutions \cite{BuckmasterVicol19} by Buckmaster and Vicol, and sharpness of the Ladyzhenskaya-Prodi-Serrin criteria \cite{CheskidovLuo22} by Cheskidov and Luo. Furthermore, the method has been also applied to the construction of non-unique dissipative Euler flows with concrete rough initial data such as vortex sheets \cite{Szekelyhidi11,MengualSzekelyhidi23} by Sz\'ekelyhidi and the third author, and recently with $C^{\nicefrac{1}{3}-\varepsilon}$ regularity \cite{EPPpp} by Enciso, Pe\~nafiel-Tom\'as and Peralta-Salas, vortex filaments \cite{GHMpp} by Gancedo, Hidalgo-Torn\'e and the third author, and also to
the inhomogeneous case \cite{GKS21} by
Gebhard, Kolumb\'{a}n and Sz\'{e}kelyhidi. These later developments are inspired by related constructions for IPM \cite{CFG11,Szekelyhidi12,CCF21,ForsterSzekelyhidi18}.

\subsection{Organization of the paper} In Section \ref{sec:sketch} we outline the proof of Theorem \ref{thm:Vishik}. In Section \ref{sec:vortex} we construct a piecewise constant unstable vortex. In Section \ref{sec:regularization} we find a regularization that is also unstable. Finally, in Sections \ref{sec:selfsimilar} and \ref{sec:nonlinear} we prove that the vortex is also self-similarly and nonlinearly unstable, respectively. 

\section{Sketch of the proof}\label{sec:sketch}

In this section we summarize the main ideas for the \hyperref[sec:proofVishik]{Proof of Theorem \ref{thm:Vishik}}. The missing  technical details will be dealt rigorously in the next sections.

Before going further let us introduce concisely Vishik's strategy to prove nonuniqueness \cite{Vishikpp1,Vishikpp2}.
Firstly, following the ideas of Jia and \v{S}ver\'{a}k \cite{JiaSverak14,JiaSverak15}, the construction of a self-similar unstable vortex $\omega_0$ gives hope to find other solutions $\omega_1$ deviating from $\omega_0$ as
$$
\|(\omega_1-\omega_0)(t)\|_{L^{\nicefrac{2}{a}}}\sim t^{\frac{\Re\lambda}{ab}},
$$
as $t\to 0$,
where $0<a,b\leq 1$ are two parameters governing the self-similar scaling, and $\lambda\in\C$ is the unstable eigenvalue ($\Re\lambda>0$). However, it is not clear how to directly construct a self-similar unstable vortex. Remarkably, Vishik realized and proved
that  some spectral properties of the linearization in self-similar coordinates can be recovered from the linearization in the original Eulerian coordinates. 

Thus, the \hyperref[sec:proofVishik]{Proof of Theorem \ref{thm:Vishik}} divides into 
  three steps:  (1) Eulerian, (2) Self-similar, and (3) Nonlinear instability. As discussed in the intro, in our approach,
Step (1) is split into two intermediate steps: (1.1) Construction of a piecewise constant unstable vortex, and (1.2) its Regularization.

\subsection*{Step 1.~Eulerian instability}\label{Step1}
We start by recalling how linear stability for the Euler equation leads to the Rayleigh stability equation, which we found convenient to express in vorticity form.

Given a background (steady) vorticity $\bar{\omega}$, a perturbation  $\bar{\omega}+\epsilon\omega$ is a second solution
to the Euler equation if the deviation $\omega$  satisfies the equation
$$
(\partial_t-L)\omega
+\epsilon (v\cdot\nabla\omega)=0.
$$
Here,  $L=L_{\bar{\omega}}$ is the linearization of the Euler equation \eqref{eq:Euler} around $\bar{\omega}$. The linearization can be written as
\begin{equation}\label{L}
L=T+K,
\end{equation}
where $T$ is a transport operator,  $K$ is  a compact operator, and they are defined by
\begin{align*}
T\omega&=-\bar{v}\cdot\nabla\omega,\\
K\omega&=-v\cdot\nabla\bar{\omega}.
\end{align*}
Here, $v$ and $\bar{v}$ are recovered from $\omega$ and $\bar{\omega}$ respectively through the Biot-Savart law \eqref{eq:BiotSavart}.

As usual in Stability theory, as $\epsilon\to 0$ one is lead to study the linear equation
\begin{equation}\label{eq:Euler:L}
(\partial_t-L)\omega^{\text{lin}}=0,
\end{equation}
and seek for solutions that grow exponentially in time. That is, for $\Re\lambda$ strictly positive, 
\begin{equation}\label{omegalin}
\omega^{\text{lin}}(t,x)=e^{\lambda t}w(x).
\end{equation}
For solutions of this form, the equation \eqref{eq:Euler:L} is equivalent to the eigenvalue problem
\begin{equation}\label{eq:L}
Lw=\lambda w.
\end{equation}
In other words, instability, at this linear level, translates into the existence of eigenvalues $\lambda$ in the right half plane
$$
\C_+=\{\lambda\in\C\,:\,\Re\lambda>0\},
$$
and eigenfunctions $w$ in some suitable Hilbert space, of the linear operator $L=L_{\bar{\omega}}$ 

\begin{Rem}\label{rem:realvalued}
It is useful to consider the eigenvalue problem \eqref{eq:L} for complex-valued vorticities. Since the operator $L$ is real, the real-valued solutions to \eqref{eq:Euler:L} will be recovered by replacing \eqref{omegalin} with
\begin{equation}\label{eq:omegalin}
\omega^{\text{lin}}(t,x)=\Re(e^{\lambda t}w(x)).
\end{equation}
Notice that the imaginary part is also a real-valued solution to \eqref{eq:Euler:L}. In particular, if $\Im\lambda=0$, both the real and imaginary parts of $w$ are eigenfunctions.
Finally, notice that if $w$ is an eigenfunction with eigenvalue $\lambda$, then its complex conjugate $w^*$ is also an eigenfunction with eigenvalue $\lambda^*$.
This remark will become relevant in the \hyperref[sec:proofVishik]{Proof of Theorem \ref{thm:Vishik}}.
\end{Rem}

The caveat of  this approach is that the spectral analysis of $L$ is  quite complicated, even for simple steady states $\bar{\omega}$. Thus,  we restrict ourselves to the case of radially symmetric $\bar{\omega}$'s called \textit{vortices}, which in polar coordinates 
$x=re^{i\theta}$ reads as
$$
\bar{\omega}(x)=\bar{w}(r).
$$

Next, we choose on which  function space the eigenvalue problem will be solved.  Given $0\neq n\in\Z$, we seek for eigenfunctions in the subspace of $L^2$ formed by purely $n$-fold symmetric vorticities 
\begin{equation}\label{eq:Un}
U_n=\{w\in L^2\,:\,w(x)=w_n(r)e^{in\theta}\}.
\end{equation}
Since $U_{-n}=U_n^*$, we will consider without loss of generality the case $n\in\N$.

\begin{defi} 
We say that the vortex $\bar{\omega}$ is \textit{unstable} if, for some $n\in\N$, there exists $0\neq w\in U_n$ satisfying \eqref{eq:L} with $\lambda\in\C_+$. 
\end{defi}

As we will see in Section \ref{sec:vortex},
the space $U_n$ is invariant under the operator $L$. Namely, for $w \in U_n$,
\begin{align*}
Tw
&=-ine^{in\theta}\frac{\bar{v}_\theta(r)}{r}w_n(r),\\
Kw
&=-ie^{in\theta}\partial_r\bar{w}(r)\int_0^\infty K_n\left(\frac{r}{s}\right)w_n(s)\dif s,
\end{align*}
where $\bar{v}_\theta$ is the angular component of the velocity $\bar{v}$, and $K_n$ is the Biot-Savart kernel acting on $U_n$ (see \eqref{eq:BSvortex} and \eqref{eq:Kn:0} respectively).

Finally, the \textit{Rayleigh stability equation} is nothing but the eigenvalue problem \eqref{eq:L} for $w\in U_n$  
\begin{equation}\label{eq:RSE:0}
\frac{\bar{v}_\theta}{r}w_n +\frac{\partial_r\bar{w}}{n}\int_0^\infty K_n\left(\frac{r}{s}\right)w_n(s)\dif s
=zw_n,
\end{equation}
where we have replaced $\lambda=-inz$.
Observe that $\Re\lambda>0$ translates to $\Im z>0$.
Notice that \eqref{eq:RSE:0} implies that $w_n$ must vanish wherever $\partial_r\bar{w}=0$. Thus, we can write
$$
w_n=h\partial_r\bar{w},
$$
in terms of some profile $h$.
For that $h$ the equation \eqref{eq:RSE:0} reads as
\begin{equation}\label{eq:RSE}
\frac{\bar{v}_\theta(r)}{r}h(r) +\frac{1}{n}\int_0^\infty K_n\left(\frac{r}{s}\right)(h\partial_r\bar{w})(s)\dif s
=zh(r),
\quad\quad
r\in\mathrm{supp}(\partial_r\bar{w}).
\end{equation}
Thus, Eulerian instability reduces to  finding a vortex that admits nontrivial  solutions $(h,z)$ of \eqref{eq:RSE} with $\Im z>0$. We will show that this is indeed the
case and therefore land in our first main result:

\begin{thm}\label{thm:L}
There exists an unstable vortex
$\bar{\Omega}\in C_c^\infty(\R^2)$ with zero mean.
\end{thm}

As anticipated in the intro, our proof of Theorem \ref{thm:L} is split into two steps:

\subsection*{Step 1.1.~Piecewise constant unstable vortex}\label{Step1.1}
In Section \ref{sec:vortex} we construct an unstable vortex of the form
$$
\bar{w}
=c 1_{[0,r_1)} - 1_{[r_1,r_2)},
$$
in terms of some parameters $0<r_1<r_2<\infty$ and $c>0$, to be determined.
Firstly, we choose $c$ making the mean of $\bar{\omega}$ equals zero
$$
\int_{\R^2}\bar{\omega}(x)\dif x
=2\pi\int_0^\infty\bar{w}(r)r\dif r
=0.
$$
This condition guarantees that $\bar{v}\in L^2$.
Moreover, $\bar{v}$ is Lipschitz and compactly supported with
$\mathrm{supp}(\bar{v})=\mathrm{supp}(\bar{\omega})=B_{r_2}.$
Secondly, since
$$
\mathrm{supp}(\partial_r\bar{w})=\{r_1,r_2\},
$$
the equation \eqref{eq:RSE} turns into two conditions for the vector $h=(h(r_1),h(r_2))$ that can be written as a linear system
\begin{equation}\label{eq:A}
A h=zh,
\end{equation}
where the matrix $A\in\R^{2\times 2}$ depends on the parameters $r_1,r_2$, and the frequency $n$.
Thus, there is an eigenvector $h\neq 0$ if and only if $z$ is a root of the characteristic polynomial $\det(A-z)$. Finally, we show in Proposition \ref{prop:eigenvalue} that for any $n\geq 2$ we can choose $r_1,r_2$ such that the roots $z,z^*$ of $\det(A-z)$ satisfy $\Im z>0$.

\begin{figure}[h!]
	\centering
	\subfigure[The vorticity profile.]{\includegraphics[width=0.40\textwidth]{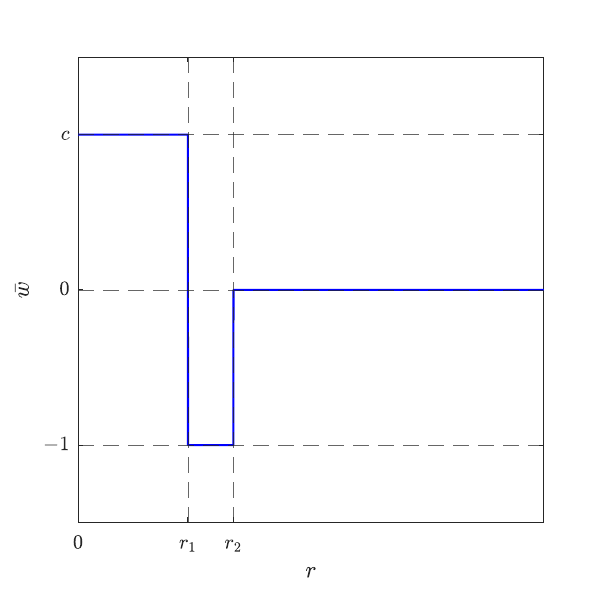}}
	\subfigure[The angular velocity profile.]{\includegraphics[width=0.40\textwidth]{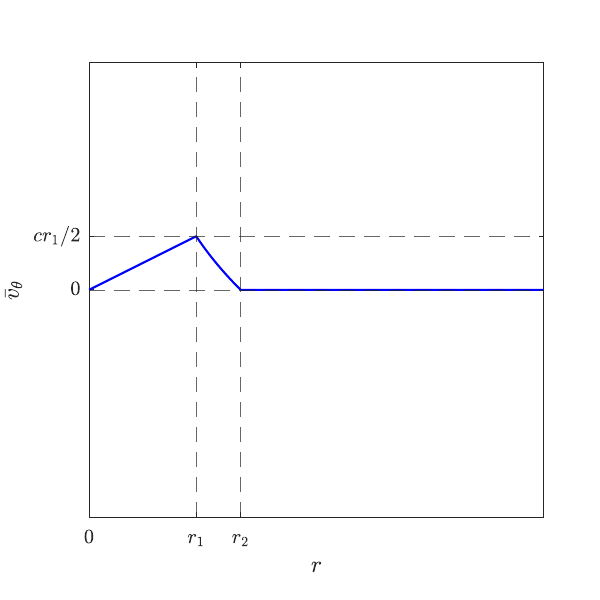}}
	\caption{The piecewise constant unstable vortex.}\label{fig:pc}
\end{figure}

\subsection*{Step 1.2.~Regularization} In Section \ref{sec:regularization} we prove that there exists a smooth vortex $\bar{\omega}^\varepsilon$, obtained by suitably regularizing $\bar{\omega}$ from \hyperref[Step1.1]{Step 1.1}, which is also unstable for some small $\varepsilon>0$. Similarly to \hyperref[Step1.1]{Step 1.1}, now we need to solve the Rayleigh stability equation \eqref{eq:RSE} in the intervals $B_\varepsilon(r_j)$ for $j=1,2$. We
rescale variables around $r_j$ writing
 $r=r_j+\varepsilon\alpha$ with $\alpha\in I=(-1,1)$.
Next, we  make the asymptotic expansions for eigenfunctions and eigenvalues. Namely, we write 
$$
h^\varepsilon(r)
=h(r_j)+\varepsilon g_j(\alpha),
\quad\quad
z^\varepsilon
=z+\varepsilon y,
$$
for some profiles $g=(g_1,g_2)\in L^2(I)^2$, and  a constant $y\in\C$, to be determined.
We also expand 
$$A^\varepsilon
g=\frac{\bar{v}_\theta^\varepsilon}{r}g+\frac{1}{n}\int_0^\infty K_n\left(\frac{r}{s}\right)(g\partial_r\bar{w}^\varepsilon)(s)\dif s
=(A+\varepsilon B)g,
$$
where $B$ is another linear operator in $L^2(I)^2$.
Therefore, the Rayleigh stability equation \eqref{eq:RSE} 
can be rewritten in the rescaled variables as
$$
(A^\varepsilon-z^\varepsilon)h^\varepsilon
=
(A-z)h
+\varepsilon(A-z)g+\varepsilon (B-y)h + \varepsilon^2 (B-y)g=0.
$$
Notice that the zero order term vanishes by \eqref{eq:A}. Since $\varepsilon>0$, 
we want to find $(g,y)$ satisfying 
\begin{equation}\label{eq:(A-z)g}
(A-z)g=(y-B)h+\varepsilon (y-B)g.
\end{equation}
The first step consists of ``inverting'' the operator $(A-z)$.
Unfortunately, since $z,z^*$ are the eigenvalues of $A$, the kernel and the image of $(A-z)$ are given by $\mathrm{Ker}(A-z)=\mathrm{span}(h)$ and $\mathrm{Im}(A-z)=\mathrm{span}(h^*)$. We bypass this obstacle by choosing $y$ in such a way that the right hand side of \eqref{eq:(A-z)g} becomes ``parallel'' to $h^*$. This allows writing \eqref{eq:(A-z)g} in the form
$$
(g,y)=H+\varepsilon R,
$$
where $H$ only depends on the $(h,z)$ solving \eqref{eq:A}, while the remainder $R$ depends also on $(g,y)$.
Thus, there is  an explicit solution for $\varepsilon=0$, and   for small enough $\varepsilon>0$  a fixed point argument in $L^2(I)^2\times\C$ yields a solution as well.
Once $\varepsilon>0$ is fixed, Theorem \ref{thm:L} holds by taking
$$
\bar{\Omega}=\bar{\omega}^\varepsilon.
$$

From now on the smooth vortex $\bar{\Omega}$ is fixed and thus we will omit $\varepsilon$ for the sake of simplicity. The use of capital letters will become clear in the following step.

\subsection*{Step 2.~Self-similar instability}\label{Step2}
Let us start by writing the Euler equation in self-similar coordinates
$$
\tau=\frac{1}{ab}\log t,
\quad\quad
X=\frac{x}{(abt)^{\nicefrac{1}{a}}},
$$
in terms of two parameters $0<a,b\leq 1$ to be determined.
Notice that the logarithmic time interval is $\R$ (instead of $\R_+$). In particular, the physical time $t=0$ corresponds to
logarithmic time $\tau=-\infty$.
It is straightforward to check that $(\omega,f)$ given by the change of variables
$$
\omega(t,x)=\frac{1}{ab t}\Omega(\tau,X),
\quad\quad
f(t,x)=\frac{1}{(ab t)^2}F(\tau,X),
$$
is a solution to the Euler equation if and only if $(\Omega,F)$ solves the \textit{self-similar Euler equation}
\begin{equation}\label{eq:Euler:SS}
\partial_\tau\Omega
+V\cdot\nabla\Omega
-bS\Omega
=F,
\end{equation}
where $S=S_a$ is the linear operator defined by
\begin{equation}\label{eq:Sdefinition}
S\Omega=\left(a+X\cdot\nabla \right)\Omega.\end{equation}
Notice that \eqref{eq:Euler:SS} agrees with the Euler equation \eqref{eq:Euler} but for the extra term $bS\Omega$, which vanishes (formally) as $b\to 0$.
The corresponding velocities and Eulerian forces and are linked by
$$
v(t,x)=(abt)^{\frac{1}{a}-1}V(\tau,X),
\quad\quad
g(t,x)=(abt)^{\frac{1}{a}-2}G(\tau,X),
$$
that is, $(v,V,g,G)$ are recovered from $(\omega,\Omega,f,F)$ through the Biot-Savart law \eqref{eq:BiotSavart}
 respectively.

\begin{Rem}
We have adopted the overall convention in \cite{ABCDGMKpp}: Given some object related to the Euler equation, the corresponding object in self-similar coordinates will be denoted with the same letter in capital case. The only exception will be \eqref{firstsolution}.
\end{Rem}

\begin{Rem}
In \cite{ABCDGMKpp} the parameters are given by $\alpha=a$ and $\beta=\frac{1}{ab}$. The reason why we consider this change of variables is purely cosmetic. For instance, it makes more manageable the spectral analysis in the limit $b\to 0$.
\end{Rem}

In the setting of Theorem \ref{thm:Vishik}, we take the background solution as the time-dependent vortex
\begin{equation}\label{firstsolution}
\omega_0(t,x)
=\frac{1}{abt}\bar{\Omega}(X),
\quad\quad
v_0(t,x)
=(abt)^{\frac{1}{a}-1}\bar{V}(X),
\end{equation}
where $\bar{\Omega}$ is the unstable vortex from Theorem \ref{thm:L}, and $\bar{V}$ the corresponding velocity  given by the Biot-Savart law  \eqref{eq:BiotSavart}.
Given $2<p<\infty$, we fix the parameter $a$ in the regime
$$
0<a p< 2,
$$
to guarantee that the integrability condition \eqref{vorticityintegrability} is satisfied. Indeed, we have the scaling
\begin{equation}\label{eq:scalingLpw0}
\|\omega_0(t)\|_{L^p}
=(abt)^{\frac{2}{a p}-1}\|\bar{\Omega}\|_{L^p},
\quad\quad
\|v_0(t)\|_{L^2}
=(abt)^{\frac{2}{a}-1}\|\bar{V}\|_{L^2}.
\end{equation}

The force is defined \textit{ad hoc} in such a way that \eqref{firstsolution} becomes a solution. By the radial symmetry of the vortex (and defining the pressure properly) the quadratic term in the Euler equation vanishes, that is, $\bar{V}\cdot\nabla\bar{\Omega}=0$. Thus, the force is defined by
$$
\partial_t \omega_0 = f,
\quad\quad
\partial_t v_0 = g,
$$
or, equivalently,
$$
F(X)
=-bS\bar{\Omega},
\quad\quad
G(X)
=b(a-S)\bar{V}.
$$
Observe that $(\bar{\Omega},\bar{V},F,G)$ are supported on $\{|X|\leq r_2+\varepsilon\}$, and that the integrability condition \eqref{forceintegrability} is also satisfied
\begin{equation}\label{eq:scalingLpf}
\int_0^t
\|f(s)\|_{L^p}\dif s
=\frac{(abt)^{\frac{2}{ap}-1}}{\frac{2}{ap}-1}
\|F\|_{L^p},
\quad\quad
\int_0^t
\|g(s)\|_{L^2}\dif s
=\frac{(abt)^{\frac{2}{a}-1}}{\frac{2}{a}-1}
\|G\|_{L^2}.
\end{equation} 

Analogously to \hyperref[Step1]{Step 1}, the deviation $\Omega$ of a close but different solution
$
\bar{\Omega}+\epsilon\Omega$
to the self similar Euler equation 
must satisfy
\begin{equation}\label{eq:Euler:lin}
(\partial_\tau-L_b)\Omega
+\epsilon (V\cdot\nabla\Omega)=0,
\end{equation}
where $L_b=L_{b,\bar{\Omega}}$ is the linearization of the self-similar Euler equation \eqref{eq:Euler:SS} around $\bar{\Omega}$
\begin{equation}\label{Lb}
L_b
=L + bS,
\end{equation}
for $S$ defined in \eqref{eq:Sdefinition}. The unstable solutions in the limit $\epsilon\to 0$ are governed by exponentially growing solutions to the linearized problem, which we denote by 
\begin{equation}\label{Omegalin}
\Omega^{\text{lin}}(\tau,X)
=e^{\lambda\tau}W(X).
\end{equation}
Hence, we need to understand the corresponding eigenvalue problem
\begin{equation}\label{eq:Lb}
L_bW=\lambda W,
\end{equation}
this time for the linear operator  $L_b=L_{b,\bar{\Omega}}$.

\begin{defi} 
We say that the vortex $\bar{\Omega}$ is \textit{self-similarly unstable} if, for some $n\in\N$, there exists $0\neq W\in U_n$ satisfying \eqref{eq:Lb} with $\lambda\in\C_+$. 
\end{defi}

An obvious but crucial observation is that
$$
e^{\lambda\tau}
=t^{\frac{\lambda}{ab}}
\to 0,
$$
as $\tau\to -\infty$ (or equivalently $t\to 0$), and therefore
\begin{equation}\label{eq:Omegalincero}
\Omega^{\text{lin}}|_{\tau=-\infty}=0.
\end{equation}

The eigenvalue problem \eqref{eq:Lb} for $W\in U_n$ is the \textit{self-similar Rayleigh stability equation}
\begin{equation}\label{eq:RSE:b}
\left(\frac{\bar{V}_\theta}{R}-\frac{b}{in}R\partial_R\right)W_n +\frac{\partial_R\bar{W}}{n}\int_0^\infty K_n\left(\frac{R}{S}\right)
W_n(S)\dif S
=zW_n,
\end{equation}
in polar coordinates $X=Re^{i\theta}$,
where we have replaced $\lambda=ab-inz$.
Unfortunately, the new term $R\partial_R W_n$ prevents from reproducing \hyperref[Step1]{Step 1} for $b>0$.
Remarkably, Vishik circumvented this obstacle by showing that the spectrum $\sigma(L_b)\cap\C_+$ converges (in a suitable sense) to $\sigma(L)\cap\C_+$ as $b\to 0$.
To this end, it is useful to decompose
\begin{equation}\label{LbExpression}
L_b=ab+T_b+K,
\end{equation}
where
\begin{align*}
T_b\Omega
&=-(\bar{V}-bX)\cdot\nabla\Omega,\\
K\Omega
&=-V\cdot\nabla\bar{\Omega},
\end{align*}
with $V=\nabla^\perp\Delta^{-1}\Omega$. 
Notice that the first term only shifts the spectrum by $ab$. 
As we will explain in Section \ref{sec:selfsimilar}, the transport operator $T_b$ is invertible in $\C_+$. Since $K$ is compact, the Operator theory allows to conclude that $\sigma(L_b)\cap\C_+$ consists of isolated eigenvalues. 
Since the dependence on $b$ is continuous, there is hope to find an eigenvalue $\lambda_b\in\sigma(L_b)\cap\C_+$ close to $\lambda\in\sigma(L)\cap\C_+$. In Section \ref{sec:selfsimilar} we will review and streamline the proof of the following result from \cite{ABCDGMKpp}:

\begin{thm}\label{thm:Lb}
The vortex $\bar{\Omega}$ from Theorem \ref{thm:L} is self-similarly unstable for some $b>0$. Moreover, the corresponding eigenfunction satisfies $W\in C_c^2(\R^2)$. 
\end{thm}

\subsection*{Step 3.~Nonlinear instability}\label{Step3}

The last step for proving Theorem \ref{thm:Vishik} requires controlling the nonlinear effects.
Coming back to \hyperref[Step2]{Step 2} (recall \eqref{eq:Euler:lin}-\eqref{eq:Lb}) we decompose the deviation into
$$
\Omega=\Omega^{\text{lin}}+\epsilon\Omega^{\text{cor}},
$$
where the correcting term $\Omega^{\text{cor}}$ must satisfy the equation
\begin{equation}\label{eq:Euler:cor}
(\partial_\tau - L_b)\Omega^{\text{cor}}
+\underbrace{(V^{\text{lin}}+\epsilon V^{\text{cor}})\cdot\nabla(\Omega^{\text{lin}}+\epsilon\Omega^{\text{cor}})}_{\mathcal{F}}
=0,
\end{equation}
coupled with the initial condition
\begin{equation}\label{eq:Omegacorinitial}
\Omega^{\text{cor}}|_{\tau=-\infty}=0.
\end{equation}
Recall that $\Omega^{\text{lin}}
=\Re(e^{\lambda\tau}W)$, while $V^{\text{lin}}$, $V^{\text{cor}}$ are the velocities associated to $\Omega^{\text{lin}}$, $\Omega^{\text{cor}}$ respectively.
If we interpret $\mathcal{F}$ as a forcing term, since the linear part decays as $e^{\Re\lambda\tau}$ and the contribution of the quadratic term is expected to be negligible for short times, one can naively expect to
gain a slightly faster exponential decay by exploiting the Duhamel formula. In Section \ref{sec:nonlinear} we will review this strategy following \cite{ABCDGMKpp} and
provide some simplifications, mainly due to the compact support of our vortex. We will prove the following result: 

\begin{thm} \label{thm:nonlinear} 
There exists a solution $\Omega^{\text{cor}}$ to \eqref{eq:Euler:cor}\eqref{eq:Omegacorinitial} satisfying
\begin{equation}\label{eq:expdecay}
\|\Omega^{\text{cor}}(\tau)\|_{L^p}
=o(e^{\Re\lambda\tau}),
\end{equation}
as $\tau\to -\infty$.
\end{thm}

\subsection*{Proof of Theorem \ref{thm:Vishik}}\label{sec:proofVishik}
We conclude this section by showing how Theorems \ref{thm:Lb} and \ref{thm:nonlinear} allow proving Theorem \ref{thm:Vishik}.
Coming back to the Eulerian coordinates, 
Theorems \ref{thm:Lb} and \ref{thm:nonlinear} imply  that
the vorticity 
$$
\omega_\epsilon
=\omega_0
+\epsilon\omega^{\text{lin}}
+\epsilon^2\omega^{\text{cor}},
$$
satisfies that $\omega_\epsilon|_{t=0}=0$, because of \eqref{eq:scalingLpw0}, \eqref{eq:Omegalincero} and \eqref{eq:Omegacorinitial}, and it is a solution to the Euler equation. We just need to  show that for different $\epsilon$'s we have different solutions.  
We  recall that
$$
\omega_0(t,x)
=\frac{1}{abt}\bar{\Omega}(X),
\quad\quad
\omega^{\text{lin}}(t,x)
=\frac{1}{abt}\Omega^{\text{lin}}(\tau,X),
\quad\quad
\omega^{\text{cor}}(t,x)
=\frac{1}{abt}\Omega^{\text{cor}}(\tau,X),
$$
where we have applied the convention that capital letters and $\tau$ stand for self-similar variables. 
Thus, by applying the scaling
$$
\|\omega^{\text{lin}}(t)\|_{L^p}
= (abt)^{\frac{2}{ap}-1}\|\Re(e^{\lambda\tau}W)\|_{L^p},
\quad\quad
\|\omega^{\text{cor}}(t)\|_{L^p}
=(abt)^{\frac{2}{ap}-1}o(e^{\Re\lambda\tau}),
$$
and the reverse triangle inequality, we obtain
$$
(abt)^{1-\frac{2}{ap}}e^{-\Re\lambda\tau}\|(\omega_\epsilon-\omega_{\bar{\epsilon}})(t)|\|_{L^p}
\geq
|\epsilon-\bar{\epsilon}|
(\|\Re(e^{i\Im\lambda\tau}W)\|_{L^p}
-o(1)).
$$
If $\Im\lambda\neq 0$, the right-hand side is positive in a sequence of times $\tau_k=\tau_0-\frac{2\pi k}{\Im z}\to -\infty$ as $k\to\infty$. If $\Im\lambda=0$, we can assume from the beginning that $W$ is real-valued. Otherwise, by Remark \ref{rem:realvalued}, it would suffice to take its imaginary part instead.
This allows concluding that
$\omega_\epsilon\neq\omega_{\bar{\epsilon}}$ whenever $\epsilon\neq\bar{\epsilon}$. Finally, we will prove within Section \ref{sec:Proofnonlinear} that the velocities $v_\epsilon=\nabla^\perp\Delta^{-1}\omega_\epsilon$ belong to $L_t^\infty L^2$.

\begin{Rem}
Notice that we have only considered $\epsilon=0,1$ in Theorem \ref{thm:Vishik}, while there are indeed infinitely many different solutions indexed by $\epsilon\geq 0$. Although these solutions depend on $\epsilon$, we have abbreviated $\Omega_\epsilon^{\text{cor}}=\Omega^{\text{cor}}$ to lighten the notation.
\end{Rem}

\begin{Rem}
In principle, the solutions we have constructed so far blow up in infinite time due to the scaling \eqref{eq:scalingLpf}. However, this can be easily solved by taking $f$ and $g$ vanishing for $t\geq 1$.
\end{Rem}

\section{Piecewise constant unstable vortex}\label{sec:vortex}

In this section we prove Theorem \ref{thm:L}.
We start by rigorously deriving the Rayleigh stability equation \eqref{eq:RSE}. 
Firstly, we write the (unforced) Euler equation
in polar coordinates $x=re^{i\theta}$
\begin{equation}\label{eq:Eulerpolar}
\partial_t\omega
+v_r\partial_r\omega
+\frac{1}{r}v_\theta\partial_\theta\omega
=0,
\end{equation}
where $v_r=v\cdot e_r$ and $v_\theta=v\cdot e_\theta$ are the radial and angular components of the velocity respectively,
in the standard basis $e_r=e^{i\theta}$ and $e_\theta=ie^{i\theta}$.

As we mentioned in Section \ref{sec:sketch}, we restrict to the case of vortices
$$
\bar{\omega}(x)
=\bar{w}(r).
$$
It is straightforward to see that the corresponding velocity  satisfies
$$
\bar{v}(x)
=\bar{v}_\theta(r)e_\theta,
$$
namely, $\mathrm{div}\,\bar{v}=0$, while
$\mathrm{curl}\,\bar{v}=\bar{\omega}$ is equivalent to
\begin{equation}\label{eq:BSvortex}
r\bar{w}=\partial_r(r\bar{v}_\theta).
\end{equation}

The linearization \eqref{L} of the Euler equation \eqref{eq:Eulerpolar} around $\bar{\omega}$ reads as
\begin{equation}\label{eq:Euler:vortex}
Tw=-\frac{1}{r}\bar{v}_\theta\partial_\theta w,
\quad\quad
Kw=-v_r\partial_r\bar{w},
\end{equation}
where $v$ is recovered from $w$ through the Biot-Savart law \eqref{eq:BiotSavart}.
Notice that only the radial component $v_r$ contributes in \eqref{eq:Euler:vortex}.
In the next lemma we compute $v_r$ for vorticities in the class \eqref{eq:Un}.

\begin{lemma}\label{lemma:vr:eigenfunction}
For any $n\geq 1$ and $w\in U_n$,
$$
v_r
=ie^{in\theta}\int_0^\infty
K_n\left(\frac{r}{s}\right)
w_n(s)\dif s,
$$
where
\begin{equation}\label{eq:Kn:0}
K_n(\rho)
=\left\lbrace
\begin{array}{cl}
	\frac{1}{2}\rho^{-1+n}, & \rho<1,\\[0.1cm]
	\frac{1}{2}\rho^{-1-n}, & \rho>1.
\end{array}
\right.
\end{equation}
\end{lemma}
\begin{proof}
By writing the Biot-Savart law \eqref{eq:BiotSavart} in polar coordinates ($x=re^{i\theta}$, $y=se^{i\vartheta}$)
$$
v(x)^*
=\frac{1}{2\pi i}\int_{\R^2}
\frac{w(y)}{x-y}\dif y
=\frac{1}{2\pi i}\int_0^\infty
\int_0^{2\pi}
\frac{w(se^{i\vartheta})}{re^{i\theta}-se^{i\vartheta}}
\dif\vartheta
s\dif s,
$$
we get
\begin{align*}
v_r
=\Re(v^*e^{i\theta})
&=\frac{\Im}{2\pi}\int_0^\infty
\int_0^{2\pi}
\frac{w(se^{i\vartheta})}{r-se^{i(\vartheta-\theta)}}\dif\vartheta
s\dif s\\
&=-\frac{1}{2\pi}\int_0^\infty \int_0^{2\pi}
\frac{\sin\vartheta}{|r-se^{i\vartheta}|^2}w(se^{i(\theta-\vartheta)})\dif\vartheta s^2\dif s.
\end{align*}
We remark that this is the expression for the real operator $v_r$ (respectively $K$). Next, we consider $v_r$ acting on complex-valued vorticities. 
Hence, for $w(x)=w_n(r)e^{in\theta}$, we have
$$
v_r
=-\frac{e^{in\theta}}{2\pi}\int_0^\infty \int_0^{2\pi}
\frac{\sin(\vartheta)e^{-in\vartheta}}{|r-se^{i\vartheta}|^2}\dif\vartheta w_n(s) s^2\dif s=ie^{in\theta}\int_0^\infty
K_n\left(\frac{r}{s}\right)
w_n(s)\dif s,
$$
where
\begin{equation}\label{eq:Kn}
K_n(\rho)
=\frac{1}{2\pi}\int_0^{2\pi}
\frac{\sin(\vartheta)\sin(n\vartheta)}{|\rho-e^{i\vartheta}|^2}\dif\vartheta
=\left\lbrace
\begin{array}{cl}
	\frac{1}{2}\rho^{-1+n}, & \rho<1,\\[0.1cm]
	\frac{1}{2}\rho^{-1-n}, & \rho>1.
\end{array}
\right.
\end{equation}
The last integral can be computed by means of the Residue Theorem (see Appendix \ref{sec:residue}).
\end{proof}

Therefore, the eigenvalue problem $Lw=\lambda w$ can be written as
\begin{equation}\label{eq:Euler:vortex:1}
\frac{\bar{v}_\theta}{r}w_n +\frac{\partial_r\bar{w}}{n}
\int_0^\infty
K_n\left(\frac{r}{s}\right)
w_n(s)\dif s
=zw_n,
\end{equation}
where $\lambda=-inz$.
Since $w_n$ must vanish wherever $\partial_r\bar{w}=0$ provided that $\Im z\neq 0$, we can write
$$
w_n
= h\partial_r\bar{w},
$$
and thus the equation \eqref{eq:Euler:vortex:1} reduces to solve
\begin{equation}\label{eq:Euler:vortex:2}
\frac{\bar{v}_\theta(r)}{r}h(r) +\frac{1}{n}
\int_0^\infty
K_n\left(\frac{r}{s}\right)
(h\partial_r\bar{w})(s)\dif s
=zh(r),
\quad\quad
r\in\mathrm{supp}(\partial_r\bar{w}),
\end{equation}
for some profile $h\neq 0$ and eigenvalue $z$ with $\Im z>0$, to be determined.

\subsection{Ansatz}

Given some parameters $0<r_1<r_2<\infty$ and $c>0$ to be determined, we consider the piecewise constant vorticity profile (see Figure \ref{fig:pc})
\begin{equation}\label{ansatz:baromega}
\bar{w}(r)
=
\left\lbrace
\begin{array}{rl}
	c, & 0<r\leq r_1, \\[0.1cm]
	-1, & r_1<r\leq r_2, \\[0.1cm]
	0, & r>r_2,
\end{array}
\right.
\end{equation}
with corresponding (angular) velocity profile \eqref{eq:BSvortex}
\begin{equation}\label{ansatz:barvphi}
\bar{v}_\theta(r)
=
\left\lbrace
\begin{array}{cl}
	\displaystyle\frac{c r}{2}, & 0<r\leq r_1, \\[0.1cm]
	\displaystyle\frac{r}{2}\left(\left(\frac{r_2}{r}\right)^2-1\right), & r_1<r\leq r_2, \\[0.1cm]
	0, & r>r_2,
\end{array}
\right.
\end{equation}
where we have chosen $c$ satisfying
\begin{equation}\label{eq:omegazeromean}
(1+c)\xi=1
\quad\text{with}\quad
\xi=\left(\frac{r_1}{r_2}\right)^2.
\end{equation}
In this case we have
$$
\partial_r\bar{w}
=-(1+c)\delta_{r_1} +\delta_{r_2}.
$$
Formally, it remains to solve the Rayleigh stability equation \eqref{eq:Euler:vortex:2} at $r=r_1,r_2$. These are two conditions for the vector $h=(h(r_1),h(r_2))$ that form the linear system
$$A
h
=zh,$$
where we split $A=D+C$ into
\begin{equation}\label{matrixA}
D=
\left[
\begin{array}{cc}
	\frac{\bar{v}_\theta(r_1)}{r_1} &  0 \\[0.1cm]
	0 & \frac{\bar{v}_\theta(r_2)}{r_2}
\end{array}
\right],
\quad\quad
C=
\left[
\begin{array}{ll}
	-\frac{1+c}{n}K_n(1) &  \frac{1}{n}K_n(\frac{r_1}{r_2}) \\[0.1cm]
	-\frac{1+c}{n}K_n(\frac{r_2}{r_1}) & \frac{1}{n}K_n(1)
\end{array}
\right].
\end{equation}
Notice that the diagonal matrix $D$ comes from the transport operator $T$, and $C$ from the compact operator $K$.

In any case, there exists an eigenvector  $h\neq 0$ for an eigenvalue $z$ if and only if $\det (A-z)=0$.  In the next lemma we compute explicitly this characteristic polynomial of $A$.

\begin{lemma}\label{lemma:detA} Let $A$ be given by \eqref{matrixA} and $r_1,r_2,c$ related to $\xi$ by \eqref{eq:omegazeromean}. Then, 
\begin{equation}\label{eq:detA}
\det (A-z)=
z^2-\frac{n-1}{n}\frac{1-\xi}{2\xi}z
+\frac{1-\xi}{4n\xi}
-\frac{1-\xi^n}{4n^2\xi}.
\end{equation}
\end{lemma}
\begin{proof}
The determinant of $A$ is a quadratic polynomial in $z$
\begin{equation}\label{eq:detA:1}
\begin{split}
\det (A-z)
&=\left(\frac{\bar{v}_\theta(r_1)}{r_1}-\frac{1+c}{n}K_n(1)-z\right)\left(\frac{\bar{v}_\theta(r_2)}{r_2}+\frac{1}{n}K_n(1)-z\right)
+\frac{1+c}{n^2}K_n\Big(\frac{r_1}{r_2}\Big)K_n\Big(\frac{r_2}{r_1}\Big)\\
&=z^2-\left(\frac{\bar{v}_\theta(r_1)}{r_1}+\frac{\bar{v}_\theta(r_2)}{r_2}-\frac{c}{n}K_n(1)\right)z\\
&+\left(\frac{\bar{v}_\theta(r_1)}{r_1}-\frac{1+c}{n}K_n(1)\right)
\left(\frac{\bar{v}_\theta(r_2)}{r_2}+\frac{1}{n}K_n(1)\right)
+\frac{1+c}{n^2}K_n\Big(\frac{r_1}{r_2}\Big)K_n\Big(\frac{r_2}{r_1}\Big).
\end{split}
\end{equation}
By plugging (recall \eqref{eq:Kn:0} and \eqref{ansatz:barvphi})
$$
\frac{\bar{v}_\theta(r_1)}{r_1}=\frac{c}{2},
\quad\quad
\frac{\bar{v}_\theta(r_2)}{r_2}=0,
\quad\quad
K_n(1)=\frac{1}{2},
\quad\quad
K_n\Big(\frac{r_1}{r_2}\Big)K_n\Big(\frac{r_2}{r_1}\Big)
=\frac{\xi^{n}}{4},
$$
and \eqref{eq:omegazeromean} into \eqref{eq:detA:1}, we get \eqref{eq:detA}.
\end{proof}

\begin{prop}\label{prop:eigenvalue}
For any $n\geq 2$ there exists $0<\xi<1$ such that the roots $z,z^*$ of the characteristic polynomial \eqref{eq:detA} satisfy $\Im z> 0$. For $n=2$ we can take $\xi=\nicefrac{1}{2}$.
\end{prop}
\begin{proof}
The discriminant $\Delta$ of the quadratic polynomial \eqref{eq:detA} equals
$$
\Delta
=\frac{p_n(\xi)}{\xi^2}
\quad\text{with}\quad
p_n(\xi)
=\left(\frac{n-1}{n}\right)^2\frac{(1-\xi)^2}{4}
-\frac{\xi(1-\xi)}{n}
+\frac{\xi(1-\xi^n)}{n^2}.
$$
Hence, it remains to show that $p_n(\xi)<0$
for some $0<\xi<1$.
By computing the first two derivatives
\begin{align*}
p_n'(\xi)
&=-\left(\frac{n-1}{n}\right)^2\frac{(1-\xi)}{2}
-\frac{(1-\xi)}{n}
+\frac{(1-\xi^n)}{n^2}
+\frac{\xi}{n}(1-\xi^{n-1}),\\
p_n''(\xi)
&=\frac{1}{2}\left(\frac{n-1}{n}\right)^2
+\frac{2}{n}(1-\xi^{n-1})
-\frac{n-1}{n}\xi^{n-1},
\end{align*}
and noticing that
$$
p_n'(1)=0,
\quad\quad
p_n''(1)
=\frac{1-n^2}{2n^2}<0,
$$
we conclude the first statement. For $n=2$ we have
$p_2(\nicefrac{1}{2})<0$.
\end{proof}

\section{Regularization}\label{sec:regularization}

Unfortunately, the unstable vortex $\bar{\omega}$ we constructed in Section \ref{sec:vortex} is not regular enough: its eigenfunction $w=e^{in\theta}h\partial_r\bar{w}$ is a measure concentrated on $r=r_1,r_2$. In this section we show that there is a regularization  $\bar{\omega}^\varepsilon$ that is also an unstable vortex for some small $\varepsilon>0$.

We take a standard mollifier $\eta\in C_c^\infty(I)$ with $I=(-1,1)$, $\int\eta=1$, and define as usual
$$
\eta^\varepsilon(\alpha)
=\frac{1}{\varepsilon}\eta\left(\frac{\alpha}{\varepsilon}\right).
$$
By mollifying $\bar{v}_\theta$ (instead of $\bar{w}$) a few computations are simplified because we get rid of some nonlocal terms.
Since we want to regularize the singularities without modifying the vortex outside $B_\varepsilon(\{r_1,r_2\})$, 
we only mollify the indicator functions.
To this end, we write the velocity \eqref{ansatz:barvphi} as
$$
\bar{v}_\theta
=\bar{v}_{\theta,1}1_{[0,r_1)}
+\bar{v}_{\theta,2}1_{[r_1,r_2)},
$$
where we have abbreviated
$$
\bar{v}_{\theta,1}
=\frac{c r}{2},
\quad\quad
\bar{v}_{\theta,2}
=\frac{r}{2}\left(\left(\frac{r_2}{r}\right)^2-1\right).
$$
Let $0<\varepsilon<\frac{1}{3}\min\{r_1,r_2-r_1\}$, to be determined.
It is straightforward to check that the profile
\begin{equation}\label{ansatz:barv:reg}
\bar{v}_\theta^\varepsilon
=\left\lbrace
\begin{array}{cl}
	\bar{v}_{\theta,1}
	+(\bar{v}_{\theta,2}-\bar{v}_{\theta,1})
	1_{[r_1,r_2)}*\eta^\varepsilon,
	& |r-r_1|<\varepsilon,\\[0.1cm]
	\bar{v}_{\theta,2}
	1_{[r_1,r_2)}*\eta^\varepsilon,
	& |r-r_2|<\varepsilon,\\[0.1cm]
	\bar{v}_\theta, & \text{otherwise},
\end{array}
\right.
\end{equation}
is smooth (and agrees with $\bar{v}_\theta$ outside $B_\varepsilon(\{r_1,r_2\})$ by definition).
The corresponding vorticity profile $\bar{w}^\varepsilon$ is then given by \eqref{eq:BSvortex}
$$
\bar{w}^\varepsilon
=\frac{1}{r}\partial_r(r\bar{v}_\theta^\varepsilon),
$$
which is smooth and agrees with $\bar{w}$ outside $B_\varepsilon(\{r_1,r_2\})$. Therefore, since
$$
\partial_r\bar{w}^\varepsilon
=0
\quad\text{outside}\quad
B_\varepsilon(\{r_1,r_2\}),
$$
it remains to find an eigenvalue $z^\varepsilon\in\C$ with $\Im z^\varepsilon>0$ and a profile $h^\varepsilon$ satisfying the Rayleigh stability equation \eqref{eq:Euler:vortex:2}
\begin{equation}\label{eq:Euler:vortex:3}
\left(\frac{\bar{v}_\theta^\varepsilon(r)}{r}-z^\varepsilon\right)h^\varepsilon(r) +\frac{1}{n}
\int_0^\infty
K_n\left(\frac{r}{s}\right)
(h^\varepsilon\partial_r\bar{w}^\varepsilon)(s)\dif s
=0,
\quad\quad
r\in B_\varepsilon(\{r_1,r_2\}).
\end{equation}

\subsection{Rescaling}

In this section we zoom into each interval $B_\varepsilon(r_j)$ through the change of variables
$$
r=r_j+\varepsilon\alpha,
$$
for $\alpha\in I=(-1,1)$ and $j=1,2$. 
From now on we will denote
\begin{equation}\label{cj}
	c_1=-(1+c),
	\quad\quad
	c_2=1,
\end{equation}
to make the notation more compact.

\begin{lemma}[Rescaling of the velocity]\label{lemma:zoom:velocity}
It holds
\begin{align*}
\frac{\bar{v}_\theta^\varepsilon(r)}{r}
=\frac{\bar{v}_\theta(r_j)}{r_j}
+\varepsilon \frac{u_j(\alpha)}{r^2},
\end{align*}
for $r=r_j+\varepsilon\alpha$, where
$$
u_j(\alpha)
=\frac{c_j}{2}\alpha(2r_j+\varepsilon\alpha)(\sigma(\alpha)-j+1)
\quad\text{with}\quad
\sigma(\alpha)
=\int_{-1}^\alpha\eta(\beta)\dif\beta.
$$
\end{lemma}
\begin{proof}
Firstly, notice that
$$
\frac{\bar{v}_{\theta,1}(r)}{r}=\frac{\bar{v}_{\theta}(r_1)}{r_1}
=\frac{c}{2},
\quad\quad
\frac{\bar{v}_{\theta}(r_2)}{r_2}=0.
$$
Hence, by computing (recall \eqref{eq:omegazeromean})
$$
\frac{\bar{v}_{\theta,2}-\bar{v}_{\theta,1}}{r}
=-\frac{1+c}{2}\frac{r^2-r_1^2}{r^2},
\quad\quad
\frac{\bar{v}_{\theta,2}}{r}
=-\frac{1}{2}\frac{r^2-r_2^2}{r^2},
$$
and
$$
1_{[r_1,r_2)}*\eta^\varepsilon
=\frac{1}{\varepsilon}\int_{r_1}^{r_2}\eta\left(\frac{r-s}{\varepsilon}\right)\dif s
=\left\lbrace
\begin{array}{cl}
\sigma(\alpha), &   r=r_1+\varepsilon\alpha, \\[0.1cm]
1-\sigma(\alpha), & r=r_2+\varepsilon\alpha,
\end{array}
\right.
$$
the lemma follows from the definitions \eqref{ansatz:barv:reg} and \eqref{cj}.
\end{proof}

\begin{lemma}[Rescaling of the vorticity]\label{lemma:zoom:vorticity}
It holds
$$
\varepsilon\partial_r\bar{w}^\varepsilon(r)
=
c_j(\bar{\eta}
+\varepsilon \zeta_j)(\alpha),
$$
for $r=r_j+\varepsilon\alpha$, where
$$
\bar{\eta}
=\eta+\partial_\alpha(\alpha\eta),
\quad\quad
\zeta_j
=-\frac{1}{2}
\partial_\alpha\left(\frac{\alpha^2\eta}{r_j+\varepsilon \alpha}\right).
$$
\end{lemma}
\begin{proof}
By applying Lemma \ref{lemma:zoom:velocity},
it is straightforward to check that
$$
\varepsilon\partial_r\bar{w}^\varepsilon
=\varepsilon\partial_r\left(\frac{1}{r}\partial_r(r\bar{v}_\theta^\varepsilon)\right)
=\partial_\alpha\left(\frac{\partial_\alpha u_j}{r_j+\varepsilon\alpha}\right).
$$
Then, by computing
$$
\frac{\partial_\alpha u_j(\alpha)}{r_j+\varepsilon\alpha}
=c_j\left(\sigma(\alpha)-j+1+\alpha\eta(\alpha)-\frac{\varepsilon}{2}\frac{ \alpha^2\eta(\alpha)}{r_j+\varepsilon\alpha}\right),
$$
we get,
$$
\partial_\alpha\left(\frac{\partial_\alpha u_j}{r_j+\varepsilon\alpha}\right)
=c_j\left(\eta+\partial_\alpha(\alpha\eta)-\frac{\varepsilon}{2}\partial_\alpha\left(\frac{\alpha^2\eta}{r_j+\varepsilon \alpha}\right)\right).
$$
This concludes the proof.
\end{proof}

We define $h^\varepsilon$ in each interval $B_\varepsilon(r_j)$ as
$$
h^\varepsilon(r)
=h_j+\varepsilon g_j(\alpha),
$$
where $h_j=h(r_j)$,
for some profiles $g_j\in L^2(I)$, to be determined. Typically, we will deal with $j=1,2$
and will speak of $g=(g_1,g_2) \in L^2(I)^2$ to deal with both equations simultaneously and make the notation more compact.  
Similarly, we define
$$
z^\varepsilon
=z+\varepsilon y,
$$
for some $y\in\C$, to be determined. 
We also make the asymptotic expansion of the kernel
\begin{equation}\label{eq:Jnjk}
K_{njk}^\varepsilon
=K_n\left(\frac{r_j+\varepsilon\alpha}{r_k+\varepsilon\beta}\right)
=K_{njk}^0
+\varepsilon J_{njk}.
\end{equation}
After all this preparation, we can write the Rayleigh stability equation in the rescaled variables.
\begin{lemma}[Rescaling of the Rayleigh stability equation \eqref{eq:Euler:vortex:3}]
\label{lemma:zoom:eq}
It holds
\begin{align*}
\left(\frac{\bar{v}_\theta^\varepsilon(r)}{r}-z^\varepsilon\right)h^\varepsilon(r) +\frac{1}{n}
\int_0^\infty
K_n\left(\frac{r}{s}\right)
(h^\varepsilon\partial_r\bar{w}^\varepsilon)(s)\dif s
=\varepsilon E_j(\alpha),
\end{align*}
for $r=r_j+\alpha\varepsilon$. 
We split the error $E=(E_1,E_2)$ into
\begin{equation}\label{eq:E=0}
E=(A-z)g+(B-y)h+\varepsilon (B-y)g,
\end{equation}
where 
\begin{align*}
(Ag)_j
&=\frac{\bar{v}_{\theta}(r_j)}{r_j}g_j
+\frac{1}{n}\sum_{k=1,2}c_kK_{njk}^0\int_I
g_k\bar{\eta}\dif\beta,\\
(Bg)_j
&=
\frac{u_j}{r^2}g_j
+\frac{1}{n}\sum_{k=1,2}c_k\int_I
J_{njk}
g_k\bar{\eta}\dif\beta
+\frac{1}{n}\sum_{k=1,2}c_k\int_I
K_{njk}^\varepsilon
g_k \zeta_k\dif\beta.
\end{align*}
\end{lemma}
\begin{proof}
On the one hand, by Lemma \ref{lemma:zoom:velocity},
\begin{equation}\label{zoom:eq:1}
\begin{split}
\left(\frac{\bar{v}_{\theta}^\varepsilon}{r}-z^\varepsilon\right)h^\varepsilon
&=\left(\frac{\bar{v}_{\theta}(r_j)}{r_j}-z\right)h_j
+\varepsilon\left(\frac{u_j}{r^2}-y\right)h_j
+\varepsilon\left(\frac{\bar{v}_{\theta}(r_j)}{r_j}-z\right)g_j
+\varepsilon^2 \left(\frac{u_j}{r^2}-y\right)g_j.
\end{split}
\end{equation}
On the other hand, by Lemma \ref{lemma:zoom:vorticity},
\begin{equation}\label{zoom:eq:2}
\begin{split}
\frac{1}{n}\int_0^\infty
K_n\left(\frac{r}{s}\right)
(h^\varepsilon\partial_r\bar{w}^\varepsilon)(s)\dif s
&=\frac{1}{n}\sum_{k=1,2} c_kh_k
K_n\left(\frac{r_j}{r_k}\right)\\
&+\frac{\varepsilon}{n}\sum_{k=1,2} c_kh_k\int_I
\frac{1}{\varepsilon}\left[K_n\left(\frac{r_j+\varepsilon\alpha}{r_k+\varepsilon\beta}\right)
-K_n\left(\frac{r_j}{r_k}\right)\right]
\bar{\eta}(\beta)\dif\beta\\
&+\frac{\varepsilon}{n}\sum_{k=1,2} c_kh_k\int_I
K_n\left(\frac{r_j+\varepsilon\alpha}{r_k+\varepsilon\beta}\right)
\zeta_k(\beta)\dif\beta\\
&+\frac{\varepsilon}{n}\sum_{k=1,2}c_k\int_I
K_n\left(\frac{r_j+\varepsilon\alpha}{r_k+\varepsilon\beta}\right)
(g_k\bar{\eta})(\beta)\dif\beta\\
&+\frac{\varepsilon^2}{n}\sum_{k=1,2}c_k\int_I
K_n\left(\frac{r_j+\varepsilon\alpha}{r_k+\varepsilon\beta}\right)
(g_k \zeta_k)(\beta)\dif\beta,
\end{split}
\end{equation}
where we have applied
$$
\int\bar{\eta}
=\int\eta=1.
$$
Finally, notice that the first terms in \eqref{zoom:eq:1} and \eqref{zoom:eq:2} cancel out because $Ah=zh$ by Section \ref{sec:vortex}.
\end{proof}

\subsection{Inverting the linear operator}

Notice that we can split $A=D+C$ into
$$
(Dg)_j=\frac{\bar{v}_\theta(r_j)}{r_j}g_j,
\quad\quad
(Cg)_j=\frac{1}{n}\sum_{k=1,2}c_kK_{njk}^0
\int_I g_k\bar{\eta}\dif\beta.
$$
Observe that when $g$ is constant we get back the matrix $A$ from \eqref{matrixA}. We omitted this detail in the introduction for the sake of simplicity. Notice that $D$ is a real-valued diagonal operator, and therefore $(D-z)$ is invertible, and that $C$ is constant-valued, that is, $C:L^2(I)^2\to\C^2$. The next lemma takes advantage of this structure to reformulate the Rayleigh stability equation $\eqref{eq:E=0}=0$ in order to find the corresponding positive eigenvalue.

\begin{lemma}\label{lem:lemaf}
The Rayleigh stability equation
$$
(A-z)g=(y-B)h+\varepsilon (y-B)g,
$$
is equivalent to
\begin{equation}\label{eq:gy}
\begin{split}
g&=f
+\gamma h^* + \delta h,\\
y&=\frac{\langle Cf,h^{*\perp}\rangle}{\langle h,h^{*\perp}\rangle},
\end{split}
\end{equation}
where $f=f(g,y)$
\begin{equation}\label{eq:f}
f=(D-z)^{-1}(-Bh+\varepsilon (y-B)g),
\end{equation}
and $\gamma,\delta\in\C$, with $\gamma=\gamma(f)$
$$
\gamma
=-\frac{1}{2iz_2|h|^2}\left(\frac{\langle h,h^*\rangle}{\langle h,h^{*\perp}\rangle}\langle Cf,h^{*\perp}\rangle
-\langle Cf,h^*\rangle\right).
$$
Since $\delta$ is arbitrary, we will take $\delta=0$ for simplicity. 
\end{lemma}
\begin{proof}
We start by making the change of variables
$$
g=f
+\mu,
$$
in terms of some $\mu$, to be determined. It turns out that $\mu$ must be a constant vector satisfying
\begin{equation}\label{eq:mu}
(A-z)\mu
=yh-Cf.
\end{equation}
Notice that $A$ in \eqref{eq:mu} is acting on constant vectors, and it is therefore given by \eqref{matrixA}.
Since $\mathrm{Ker}(A-z)=\mathrm{span}(h)$, $\mathrm{Im}(A-z)=\mathrm{span}(h^*)$, and $\{h,h^*\}$ is a basis of $\C^2$,
we take
$$
\mu=\gamma h^*+\delta h,
$$
in terms of some $\gamma,\delta\in\C$, to be determined. Notice that
\begin{equation}\label{eq:Amenoszrange}
(A-z)\mu
=\gamma(A-z)h^*
=\gamma(A-(z^*+2iz_2))h^*
=-2i z_2\gamma h^*.
\end{equation}
Hence, 
multiplying \eqref{eq:mu} by $h^{*\perp}=(-h_2^*,h_1^*)$, we get the compatibility condition for $y$
$$
0
=y\langle h,h^{*\perp}\rangle
-\langle Cf,h^{*\perp}\rangle.
$$
Notice that $\langle h,h^{*\perp}\rangle\neq 0$ since $\{h,h^*\}$ is a basis of $\C^2$. On the other hand, multiplying \eqref{eq:mu} by $h^*$, and using \eqref{eq:Amenoszrange},
we can determine $\gamma$ from the following equality
$$
-2iz_2\gamma
|h|^2
=y\langle h,h^*\rangle
-\langle Cf,h^*\rangle.
$$
The proof of the lemma is concluded. 
\end{proof}

\subsection{Fixed point argument}

For $\varepsilon=0$ the equation \eqref{eq:gy} for $(g,y)$ is explicit and thus we can find a solution $(g^0,y^0)$. For small $\varepsilon>0$ we will apply a fixed point argument. 
It can be shown by a bootstrapping argument that the eigenfunction $g^\varepsilon$ is indeed smooth. However, since this additional information is not necessary at this point, we omit the proof.

\begin{prop} 
For every $M>\|(g^0,y^0)\|_{L^2(I)^2\times\C}$ there exists $\varepsilon_0>0$ satisfying that: for any $0\leq\varepsilon\leq\varepsilon_0$ there exists $(g^\varepsilon,y^\varepsilon)\in L^2(I)^2\times\C$ solving \eqref{eq:gy} with 
$\|(g^\varepsilon,y^\varepsilon)\|_{L^2(I)^2\times\C}\leq M$. In particular, $\Im z^\varepsilon>0$ for $\varepsilon<\frac{\Im z}{M}$.
\end{prop}
\begin{proof}
First we check that the linear operators $B$ and $C$
are bounded in $L^2(I)^2$. The operator $C$ can be easily bounded by
$$
\|Cg\|_{L^2}
\leq\frac{|I|^{\nicefrac{1}{2}}}{n}
\sum_{j=1,2}\sum_{k=1,2}|c_k||K_{njk}^0|
\|g_k\|_{L^2}\|\bar{\eta}\|_{L^2}.
$$
Similarly, the operator $B$ can be bounded by
$$
\|Bg\|_{L^2}
\leq
\sum_{j=1,2}\left(
\|\frac{u_j}{r^2}\|_{L^\infty}\|g_j\|_{L^2}
+\frac{|I|^{\nicefrac{1}{2}}}{n}\sum_{k=1,2}|c_k|
\left(
\|J_{njk}\|_{L^\infty}
\|\bar{\eta}\|_{L^\infty}
+\|K_{njk}^\varepsilon\|_{L^\infty}
\|\zeta_k\|_{L^\infty}
\right)\|g_k\|_{L^2}\right).
$$
From the expression \eqref{eq:Kn:0} and \eqref{eq:Jnjk} it follows that
$$
\|K_{njk}^\varepsilon\|_{L^\infty}
\leq\frac{1}{2},
\quad\quad
\|J_{njk}\|_{L^\infty}
\leq\frac{1}{2}(n+1)
\frac{r_j+r_k}{(r_k-\varepsilon)^2}.$$
Let us denote by $F$ the map on the right hand side of \eqref{eq:gy}, that is, we rewrite this equation compactly as
$$
(g,y)=F(g,y).
$$
On the one hand, since all the operators involved in $F$ are bounded and $\|(g^0,y^0)\|_{L^2(I)^2\times\C}<M$, there exists $\varepsilon_0>0$ such that $F$ maps the ball $B_M$ of $L^2(I)^2\times\C$ into itself.
On the other hand, given a pair $(g,y)$ and $(\bar{g},\bar{y})$ in $B_M$,  the  corresponding pair
$f$ and $\bar{f}$ given by \eqref{eq:f} satisfy
\begin{equation}\label{eq:contr}
f-\bar{f}
=\varepsilon(D-z)^{-1}((y-B)g-(\bar{y}-B)\bar{g}),
\end{equation}
and thus
$$
\|f-\bar{f}\|_{L^2}
\leq\varepsilon
\|(D-z)^{-1}\|_{L^\infty}
\left(|y-\bar{y}|\|g\|_{L^2(I)^2}
+|\bar{y}|\|g-\bar{g}\|_{L^2}
+\|B(g-\bar{g})\|_{L^2(I)^2}\right).
$$
The boundedness of the operators $B$ and $C$ and the constant $M$ in the statement of the proposition imply that we can choose $\varepsilon_0>0$ such that $F$ becomes a contraction on $B_M$. Then, we can apply the classical Banach fixed point theorem to find our required solution. 
\end{proof}

\section{Self-similar instability}\label{sec:selfsimilar}

In this section we prove Theorem \ref{thm:L}, and provide growth bounds of the semigroup generated by $L_b$.
We start by decomposing
$$
L_b
=
ab
+T_b+K,
$$
where $T_b$ is the transport operator
$$
T_b\Omega
=-\bar{V}_b\cdot\nabla\Omega
\quad\text{with}\quad
\bar{V}_b
=\bar{V}-b X,
$$
and $K$ is the compact operator
$$
K\Omega=-V\cdot\nabla\bar{\Omega}
\quad\text{with}\quad
V=\nabla^\perp\Delta^{-1}\Omega.
$$

\begin{Rem}
As these operators arise from the self-similar Euler equation, we will keep the convention of using uppercase letters for maintaining coherence in notation. It is important to observe that the results applicable to $b=0$ correspond to the (original) Euler equation.
\end{Rem}

From now on we consider $n\geq 2$ fixed.
We define these operators on the subspace $L_n^2$ of vorticities $\Omega\in L^2$ which has zero mean and are $n$-fold symmetric 
\begin{equation}\label{eq:nfold}
\Omega(X)=\Omega(e^{\frac{2\pi i}{n}}X).
\end{equation} 
By writing the Fourier expansion of $\Omega$ in polar coordinates $X=Re^{i\theta}$,
$$
\Omega(X)=\sum_{k\in\Z}\Omega_k(R)e^{ik\theta},
\quad\quad
\Omega_k(R)
=\frac{1}{2\pi}\int_0^{2\pi}\Omega(Re^{i\theta})e^{-ik\theta}\dif\theta,
$$
it follows that
the $n$-fold symmetry \eqref{eq:nfold} is equivalent to the vanishing of the indices $k$ that are not multiples of $n$, and the zero mean condition to
\begin{equation}\label{eq:U0}
\int_0^\infty \Omega_0(R)R\dif R=0.
\end{equation}
Therefore, we can decompose $L_n^2$ into the orthogonal direct sum
\begin{equation}\label{eq:directsum}
L_n^2=\bigoplus_{j\in\Z}U_{jn},
\end{equation}
of the invariant subspaces given in \eqref{eq:Un}, while $U_0$ is given by the condition \eqref{eq:U0}.
More precisely, by the Plancherel identity
\begin{equation}\label{eq:L2n}
\|\Omega\|_{L^2}^2
=2\pi
\sum_{j\in\Z}
\int_0^\infty |\Omega_{jn}(R)|^2R\dif R,
\end{equation}
we consider sums of elements $\Omega_{jn}$ for which \eqref{eq:L2n} is finite.
We remark that the operators under consideration are closed and densely defined. In fact, the domain of $K$ is $D(K)=L_n^2$, and the domain of $L_b$ and $T_b$ equals
$$
D(L_b)=D(T_b)
=\{\Omega\in L_n^2\,:\,\mathrm{div}(\bar{V}_b\Omega)\in L_n^2\}.
$$

\begin{Rem}
The reason for considering the direct sum \eqref{eq:directsum} is that the nonlinear term in the Euler equation lacks invariance in $U_n$, but it is invariant in the entire space $L_n^2$ (see Proposition \ref{propABCDGMKpp}). 
Specifically, in Section \ref{sec:nonlinear} we will need the growth bound of the semigroup generated by $L_b$ acting on the full space $L_n^2$. 
\end{Rem}

Before embarking on the proof of Theorem \ref{thm:Lb}, we recall some classical results in Operator theory. Then, we will analyze $T_b$ and $K$ separately, and later $L_b$. After that, we will compute the growth bound of the semigroup generated by $L_b$, and provide regularity properties of the eigenfunction.

\subsection{Preliminaries}

In this section we recall some classical results in Operator theory that will be helpful during the analysis. In general, we will consider a linear operator $A:D(A)\subset H\to H$ acting on some Hilbert space $H$, where $D(A)$ is the domain of $A$. For a fixed $H$, we will denote $\mathcal{L}$ and $\mathcal{K}$ by the space of bounded and compact operators on $H$ respectively. In the next sections we will consider $H=L_n^2$.

\subsubsection{Spectral theory}
The \textit{spectrum} of $A$ is defined as  
$$
\sigma(A)=\{\lambda\in\C\,:\, (A-\lambda)\textit{ is not invertible}\}.
$$
Let us suppose that $A$ is a  bounded operator. 
Then, $A$ is called a  \textit{Fredholm Operator} if both the kernel $\text{Ker}(A)$ and the cokernel $H/\text{Im}(A)$ are finite dimensional. In this case, the \textit{index} of $A$ is the integer
$$
\text{Ind}(A)
=\dim(\text{Ker}(A))
-\dim(H/\text{Im}(A)).
$$
We recall a classical result in Spectral theory: the stability of the index of Fredholm operators with respect to compact perturbations
(see e.g.~\cite[Theorem 5.26, Chapter IV]{Kato95}).

\begin{prop}\label{prop:StabilityFredholm}
	Let $A\in\mathcal{L}$ be a Fredholm operator and $K\in\mathcal{K}$. Then, $A+K\in\mathcal{L}$ is a Fredholm operator with $\text{Ind}(A+K)=\text{Ind}(A)$.
\end{prop}


\subsubsection{Semigroup theory}
We recall several classical results in Semigroup theory. The first one gives a characterization of strongly continuous semigroups (see e.g.~\cite[Corollary 3.6, Chapter II]{EngellNagel00}).
\begin{prop}\label{prop:semigroupTFAE} 
Given $w\in\R$ and a linear operator $A$, the following are equivalent:
\begin{enumerate}[(i)]
    \item $A$ generates a strongly continuous semigroup satisfying $\|e^{sA}\|_{\mathcal{L}}\leq e^{ws}$ for all $s\geq 0$.
    \item $A$ is closed, densely defined, and for any $\lambda\in\C$ with $\Re\lambda>w$, it holds
    $$
    \|(A-\lambda)^{-1}\|_{\mathcal{L}}\leq\frac{1}{\Re\lambda - w}.
    $$
\end{enumerate}
\end{prop}

Next, we recall that, given a strongly continuous semigroup $e^{sA}$ generated by an operator $A$,
its \textit{growth bound} is defined as
\begin{equation}\label{eq:growthbound}
\omega_0(A)
=\inf\{w\in\R\,:\,\text{there exists }C_w\geq 1\text{ such that }\|e^{s A}\|_{\mathcal{L}}
\leq C_w e^{w s}\text{ for all }s\geq 0\}.
\end{equation}
The second result from semigroup theory that we need (see e.g.~\cite[Corollary 2.11, Chapter IV]{EngellNagel00}) relates \eqref{eq:growthbound} with the \textit{essential bound} 
$$
\omega_{\text{ess}}(A)
=\inf_{s>0}\frac{1}{s}\log\|e^{s A}\|_{\mathcal{L}/\mathcal{K}},
$$
where $\mathcal{L}/\mathcal{K}$ is the quotient space (the so-called Calkin algebra),
and the \textit{spectral bound}
$$
s(A)
=\sup\{\Re\lambda\,:\,\lambda\in\sigma(A)\}.
$$

\begin{prop}\label{prop:growthbound}
Given an operator $A$ generating a strongly continuous semigroup, it holds
$$
\omega_0(A)
=\max\{\omega_{\text{ess}}(A),s(A)\}.
$$
Moreover,
$
\sigma(A)\cap \{\Re\lambda\geq w\}
$
is finite for any $w>\omega_{\text{ess}}(A)$.
\end{prop}

The third result in semigroup theory that we need is the stability of the essential bound with
respect to compact perturbations (see e.g.~\cite[Proposition 2.12, Chapter IV]{EngellNagel00}).

\begin{prop}\label{prop:wessA+K}
Given an operator $A$ generating a strongly continuous semigroup, and $K\in\mathcal{K},$ it holds
$$
\omega_{\text{ess}}(A+K)
=\omega_{\text{ess}}(A).
$$
\end{prop}

\subsection{The transport operator $T_b$}


\begin{lemma}\label{lemma:Tbsemigroup}
	The operator $T_b$ generates a strongly continuous
 semigroup
	$$
	e^{sT_b}
	\Omega
	=\Omega\circ\bar{X}_b(s,\cdot)^{-1},
	$$
	where $\bar{X}_b$ is the flow map
\begin{equation}\label{eq:flowmap}
\partial_s\bar{X}_b
	=\bar{V}_b(\bar{X}_b),
	\quad\quad
	\bar{X}_b|_{s=0}=\text{id}.
\end{equation}
	Moreover,
\begin{equation}\label{eq:esTbbound}
\|e^{sT_b}\|_{\mathcal{L}}
	=e^{-bs}.
\end{equation}
\end{lemma}
\begin{proof}
Notice that the flow map is defined through a Lipschitz vector field
$$
\|\bar{V}_b\|_{Lip}
\leq\|\bar{V}\|_{Lip}
+b.
$$
Hence, the first statement follows from the Cauchy-Lipschitz theory applied to the transport equation.
For the second statement, by solving the ODE
$$
\partial_sJ_{\bar{X}_b}=\mathrm{div}(\bar{V}_b)J_{\bar{X}_b}
\quad\text{with}\quad
\mathrm{div}(\bar{V}_b)=-2b,
$$
we deduce that the Jacobian of the flow map equals
\begin{equation}\label{JX}
J_{\bar{X}_b}=e^{-2bs}.
\end{equation}
Therefore,
$$
\int_{\R^2}|e^{sT_b}\Omega|^2\dif X
=e^{-2bs}
\int_{\R^2}|\Omega|^2\dif Y,
$$
where $X=\bar{X}_b(s,Y)$.
\end{proof}

\begin{lemma}\label{lemma:Tinv}
The resolvent map
\begin{equation}\label{eq:Tinv}
\begin{split}
(T_b-\lambda)^{-1}:
\{\Re\lambda>-b\}& \rightarrow  \mathcal{L} \\
(b,\lambda) & \mapsto
\left(\Omega\mapsto-\int_0^\infty
e^{s(T_b-\lambda)}\Omega\dif s\right)
\end{split}
\end{equation}
is well defined with
\begin{equation}\label{eq:resolventbound}
\|(T_b-\lambda)^{-1}\|_{\mathcal{L}}\leq
\frac{1}{\Re\lambda+b}.
\end{equation}
For any $\Omega\in L_n^2$, the map
$(b,\lambda)\mapsto (T_b-\lambda)^{-1}\Omega$ is continuous.
\end{lemma}
\begin{proof}
The first statement follows directly from Lemma \ref{lemma:Tbsemigroup} and Proposition \ref{prop:semigroupTFAE}.
In fact, a simple integration by parts shows that \eqref{eq:Tinv} is the inverse of $(T_b-\lambda)$.
Alternatively, \eqref{eq:Tinv} is the Laplace transform of the semigroup $e^{sT_b}$. 
For the second statement, given $\Omega\in L_n^2\cap C_c^\infty$ and $s>0$, the map $(b,\lambda)\mapsto e^{s(T_b-\lambda)}\Omega$ is continuous since the flow map \eqref{eq:flowmap} is defined through a continuous in $b$ uniformly Lipschitz vector field. The continuity (indeed analiticity) in $\lambda$ is clear from the definition.
Hence, \eqref{eq:esTbbound} allows to apply the dominated convergence theorem in \eqref{eq:Tinv}. Finally, by applying that $L_n^2\cap C_c^\infty$ is strongly dense and \eqref{eq:resolventbound}, 
we conclude the proof.
\end{proof}

\subsection{The compact operator $K$}

As pointed in \cite{ABCDGMKpp}, while it is not feasible to extend the Biot-Savart operator throughout $L^2$, it can be achieved within the subspace $L_n^2$, thanks to the following proposition.  By a slight abuse of notation, we will continue to denote the Biot-Savart operator as $V=\nabla^\perp\Delta^{-1}\Omega$.
By combining \eqref{eq:VDVL2} with the Rellich-Kondrachov Theorem, and the fact that $\bar{\Omega}$ is compactly supported, we deduce the compactness of $K$ as a corollary. The second estimate \eqref{eq:VL2} will imply that the velocity of the eigenfunction belongs to $L^2$.

\begin{lemma}\label{lemma:VDVL2}
There exists $C>0$ such that
\begin{equation}\label{eq:VDVL2}
\frac{1}{R}\|V\|_{L^2(B_R)}+\|DV\|_{L^2}
\leq C\|\Omega\|_{L^2},
\end{equation}
for any $R>0$ and $\Omega\in L_n^2$, where $V=\nabla^\perp\Delta^{-1}\Omega$. 
If $\mathrm{supp}(\Omega)\subset B_{\bar{R}}$ for some $\bar{R}>0$, then also
\begin{equation}\label{eq:VL2}
\|V\|_{L^2}\leq C\bar{R}\|\Omega\|_{L^2}.
\end{equation}
\end{lemma}
\begin{proof}
Let $\Omega\in L_n^2\cap C_c$. 
Firstly, by applying the $n$-fold symmetry \eqref{eq:nfold}, we deduce that
\begin{equation}\label{eq:Vnfold}
V(X)
=\left(\frac{1}{2\pi i}\int_{\R^2}\frac{\Omega(Y)}{X-Y}\dif Y\right)^*
=\left(\frac{1}{2\pi i}\int_{\R^2}\frac{\Omega(Y)}{X-e^{-\frac{2\pi i k}{n}}Y}\dif Y\right)^*
=e^{-\frac{2\pi i k}{n}} V(e^{\frac{2\pi i k}{n}}X).
\end{equation}
Therefore, by integrating $V$ on the ball $B_R$, we get
$$
\int_{B_R}V(x)\dif x
=e^{-\frac{2\pi i k}{n}}\int_{B_R}V(x)\dif x,
$$
from which we deduce that
\begin{equation}\label{eq:intV=0}
\int_{B_R}V(x)\dif x
=\left(\frac{1}{n}\sum_{k=0}^{n-1}e^{-\frac{2\pi i k}{n}}\right)
\int_{B_R}V(x)\dif x
= 0.
\end{equation}
Therefore, by applying the Poincar\'e inequality, we get
\begin{equation}\label{eq:Poincare}
\|V\|_{L^2(B_R)}
\leq CR\|DV\|_{L^2(B_R)}.
\end{equation}
Secondly, by applying the Plancherel identity on the Biot-Savart law \eqref{eq:BiotSavart}, we obtain
$$\|DV\|_{L^2}=\|\Omega\|_{L^2}.$$ 
These inequalities allow to extend the Biot-Savart operator in $L_n^2$ by density. For the second statement, since we already have \eqref{eq:Poincare} for $R=2\bar{R}$, it is enough to estimate the $L^2$-norm of $V$ outside $B_{2\bar{R}}$. By applying the zero-mean condition of $\Omega$, we get
\begin{align*}
\|V\|_{L^2(\R^2\setminus B_{2\bar{R}})}
&=\frac{1}{2\pi}\left(\int_{\R^2\setminus B_{2\bar{R}}}\left|\int_{B_{\bar{R}}}\Omega(y)\left(\frac{1}{X-Y}-\frac{1}{X}\right)\dif Y\right|^2\dif X\right)^{\nicefrac{1}{2}}\\
&\leq \frac{C}{\bar{R}}
\int_{B_{\bar{R}}}|\Omega (Y)||Y|\dif Y
\leq C\bar{R}\|\Omega\|_{L^2},
\end{align*}
where we have applied $|X-Y|\geq\frac{1}{2}|X|$ in the first inequality, and the H\"older inequality in the second one. The constant $C>0$ changes from line to line, but it keeps universal.
\end{proof}

\begin{cor}
The operator $K$ is compact.
\end{cor}

\subsection{The operator $L_b$}

\begin{lemma}\label{lemma:Fredholm}
For any $\lambda\in\C$ with $\Re\lambda>b(a-1)$, the operator $(L_b-\lambda)$ is Fredholm with index zero.
Therefore, $\sigma(L_b)\cap\{\Re\lambda> b(a-1)\}$ consists of eigenvalues with finite multiplicity.
\end{lemma}
\begin{proof}
We split
\begin{equation}\label{eq:Llambdasplit}
L_b-\lambda
=(T_b-(\lambda - ab))+K.
\end{equation}
Since $\Re\lambda>b(a-1)$, the first term is invertible by Lemma \ref{lemma:Tinv}. In particular, it is a Fredholm operator with index zero.
Since $K$ is compact, the operator \eqref{eq:Llambdasplit} is also Fredholm with index zero (in particular positive) by Proposition \ref{prop:StabilityFredholm}. Thus the second claim follows.
\end{proof}

\begin{lemma} The operator $L_b$ generates a strongly continuous semigroup.
\end{lemma}
\begin{proof}
By applying Lemma \ref{lemma:Tinv} and that $K$ is compact, we deduce that the map
\begin{equation}\label{eq:C}
\begin{split}
	C:
	\{\Re\lambda>b(a-1)\} & \rightarrow  \mathcal{K} \\
	(b,\lambda) & \mapsto
	(T_b-(\lambda - ab))^{-1}\circ K,
\end{split}
\end{equation}
is well defined and continuous. 
In fact, 
the continuity of \eqref{eq:C} in the operator norm follows from combining the fact that the resolvent map \eqref{eq:Tinv} is continuous for any fixed $\Omega$, and that $K$ can be approximated by finite rank operators. 
The map $C$ allows decomposing \eqref{eq:Llambdasplit} into
\begin{equation}\label{eq:composition} 
L_b-\lambda
=(T_b-(\lambda - ab))\circ(I+C(b,\lambda)).
\end{equation}
Let $\lambda\in\C$ with $\Re\lambda+b(1-a)\geq 2\|K\|_{\mathcal{L}}$. By applying \eqref{eq:resolventbound}, we deduce that
$$
\|C(b,\lambda)\|_{\mathcal{L}}\leq\frac{\|K\|_{\mathcal{L}}}{\Re\lambda +b(1-a)}\leq\frac{1}{2},
$$
and thus the Neumann series
$$
(I+C(b,\lambda))^{-1}
=\sum^\infty_{n=1} (-C(b,\lambda))^n,$$
converges in the operator norm. Hence, $I+C(b,\lambda)$ is invertible with the bound
\begin{equation}\label{eq:C1/2}
\|(I+C(b,\lambda))^{-1}\|_{\mathcal{L}}\leq 1.
\end{equation}
Finally, by applying again \eqref{eq:resolventbound}, combined with \eqref{eq:composition} and \eqref{eq:C1/2}, we deduce that
$$
\|(L_b-\lambda)^{-1}\|_{\mathcal{L}}
\leq\frac{1}{\Re\lambda + b(1-a)}
\leq\frac{1}{\Re\lambda + b(1-a)-2\|K\|_{\mathcal{L}}},
$$
for any $\Re\lambda + b(1-a)\geq 2\|K\|_{\mathcal{L}}$. Then, the claim follows by applying Proposition \ref{prop:semigroupTFAE}.
\end{proof}

\begin{lemma}\label{lemma:essentialbound}
The essential bound
satisfies 
$$
\omega_{\text{ess}}(L_b)\leq (a-1)b\leq 0.$$
Therefore, $\sigma(L_b)\cap\{\Re\lambda>w\}$ is finite for any $w> (a-1)b$.
\end{lemma}
\begin{proof}
Since the identity commutes with any operator, we can decompose 
$$
e^{sL_b}
=e^{sab}e^{s(T_b+K)}.
$$
Thus,
$$
\omega_{\text{ess}}(L_b)
\leq ab+\omega_{\text{ess}}(T_b+K).
$$
Then, by applying 
\eqref{eq:esTbbound}, that $K$ is compact, and Proposition \ref{prop:wessA+K}, we get that 
$$
\omega_{\text{ess}}(T_b+K)
=\omega_{\text{ess}}(T_b)
= -b.
$$
The second part of the statement follows from Proposition \ref{prop:growthbound}.
\end{proof}

\subsection{Proof of Theorem \ref{thm:Lb}}
\label{sec:ProofthmLb}
Firstly, we claim that there exists $b>0$ satisfying
\begin{equation}\label{eq:PropLb}
\sigma(L_b)\cap\{\Re\lambda>3b\}\neq\emptyset.
\end{equation}
Since $L=L_0$, we know from Theorem \ref{thm:L} that \eqref{eq:PropLb} holds for $b=0$. More precisely, there exists an eigenvalue $\lambda_0\in\sigma(L)\cap\C_+$ and an eigenfunction $0\neq W_0\in\text{Ker}(L-\lambda_0)$.
We will prove by contradiction that necessarily \eqref{eq:PropLb} is true for some $0<b\leq\frac{\Re\lambda_0}{4}$.

Let us suppose that \eqref{eq:PropLb} is false. Thus, $(L_b-\lambda)$ is invertible for all $\lambda \in  \C_+$ and $0<b\leq\frac{\Re\lambda_0}{4}$. We will first prove that in fact $(L_b-\lambda)^{-1}$ is continuous as a function of $(b, \lambda)$.

By applying Lemma \ref{lemma:essentialbound}  to $b=0$, we can take a $\circlearrowleft$-oriented circle $\Gamma:\T\to\C_+\setminus\sigma(L)$ of radius strictly less than $\frac{\Re\lambda_0}{4}$ surrounding $\lambda_0$. By applying Lemma \ref{lemma:Tinv} coupled with \eqref{eq:C} and \eqref{eq:composition}, our hypothesis implies that $I+C$ maps continuously the compact set $[0,\frac{\Re\lambda_0}{4}]\times\Gamma(\T)$ into 
$$
\{A\in\mathcal{L}\,:\,A\text{ invertible}\}
=\bigcup_{N\in\N}
\{A\in\mathcal{L}\,:\,\|A^{-1}\|_{\mathcal{L}}< N\}.
$$
Since the last union forms an open cover, we deduce that there exists $N\in\N$ such that
$$
\|(I+C(b,\lambda))^{-1}\|_{\mathcal{L}}\leq N,
$$
uniformly in $(b,\lambda)\in [0,\frac{\Re\lambda_0}{4}]\times\Gamma(\T)$, 
and thus, by applying \eqref{eq:resolventbound} and \eqref{eq:composition}, also
$$
\|(L_b-\lambda)^{-1}\|_{\mathcal{L}}
\leq \frac{N}{\Re\lambda + b(1-a)}.
$$
By applying these bounds and the resolvent identity
\begin{align*}
(I+C(b,\lambda))^{-1}-(I+C(b',\lambda'))^{-1}
=(I+C(b,\lambda))^{-1}\circ(C(b,\lambda)-C(b',\lambda'))\circ(I+C(b',\lambda'))^{-1},
\end{align*}
it follows that 
the map $(b,\lambda)\mapsto (I+C(b,\lambda))^{-1}$ is continuous from $[0,\frac{\Re\lambda_0}{4}]\times\Gamma(\T)$ into $\mathcal{L}$. As a consequence, the same can be deduce for the map
$(b,\lambda)\mapsto (L_b-\lambda)^{-1}$.

Once continuity is obtained, we can consider the Riesz projection
$$
P_{b,\Gamma}=-\frac{1}{2\pi i}\int_\Gamma (L_b-\lambda)^{-1}\dif\lambda.
$$
and take limits under the integral sign by the dominated convergence theorem. Therefore, we have
$$
P_{b,\Gamma}(W_0)
\to P_{0,\Gamma}(W_0),
$$
as $b\to 0$. However, $P_{b,\Gamma}(W_0)=0$ for any $0<b\leq\frac{\Re\lambda_0}{4}$ by hypothesis, while $P_{0,\Gamma}(W_0)=W_0\neq 0$, which is a contradiction. 

Therefore, there exists $\lambda_b\in\sigma(L_b)\cap\{\Re\lambda>3b\}$ for some $0<b\leq\frac{\Re\lambda_0}{4}$. 
By Lemma \ref{lemma:Fredholm}, there exists $0\neq W_b\in\mathrm{Ker}(L_b-\lambda_b)$.
Since $L_b$ is invariant in each $U_{jn}$, the functions $W_{b,jn}(R)e^{ijn \theta}$ are also eigenfunctions. Furthermore, since $\mathrm{Ker}(L_b-\lambda_b)$ is finite dimensional, $W_{b,jn}$ is null
for all but a finite number of $j$'s. For $j=0$, since $L_b W_{b,0}=\lambda_b W_{b,0}$ reads as
$$
b(a+R\partial_R)W_{b,0}=\lambda_b W_{b,0},
$$
necessarily
$W_{b,0}=0$. Therefore, there exists 
$
0\neq W_{b,jn}(r)e^{ijn\theta}\in\mathrm{Ker}(L_b-\lambda_b)\cap U_{jn}
$
for some $j\neq 0$.

\subsection{Growth bound}
In the \hyperref[sec:ProofthmLb]{Proof of Theorem \ref{thm:Lb}} we have seen the existence of an eigenvalue $\lambda\in\sigma(L_b)\cap\{\Re\lambda>3b\}$. 
In this section we select the eigenvalue with largest real part, and we compute the growth bound of $L_b$. In the section about the nonlinear instability we will need the estimate for the growth of the semigroup norm. 

\begin{prop}\label{prop:GrowthBound}
There exists $\lambda\in\sigma(L_b)\cap\{\Re\lambda>3b\}$ such that $\omega_0(L_b)=s(L_b)=\Re\lambda$.
Therefore, for any $\delta>0$ there exists $C_\delta\geq 1$ such that
$$
\|e^{s L_b}\|_{\mathcal{L}}
\leq C_\delta e^{(\Re\lambda+\delta)s},
$$
for all $s\geq 0$.
\end{prop}
\begin{proof} By Section \ref{sec:ProofthmLb} we know that there exists $\lambda_b\in\sigma(L_b)\cap\{\Re\lambda>3b\}$. 
Hence, by combining Proposition  \ref{prop:growthbound} and Lemma \ref{lemma:essentialbound}, we deduce the statement.
\end{proof}

\subsection{The eigenfunction}

In this section we prove that the eigenfunction associated to the eigenvalue $\lambda$ appearing in Proposition \ref{prop:GrowthBound} is smooth and compactly supported. Here we see, once again,  the advantages of writing the Rayleigh stability equation in vorticity form and dealing with vortices with compact support.

\begin{prop}\label{prop:eigenfunction}
Let $W\in\mathrm{Ker}(L_b-\lambda)\cap U_{jn}$ with $j\neq 0$. Then, 
$$
W\in C^\infty(\R^2\setminus\{0\})
\cap C_c^\gamma(\R^2),
$$
with
$\gamma =(\frac{\Re\lambda}{b}-a)\geq 2$. 
\end{prop}
\begin{proof}
The self-similar Rayleigh stability equation \eqref{eq:RSE:b} can be rewritten as
\begin{equation}\label{eq:RSE:eigenfunction}
\left((\lambda-ab)+ijn\frac{\bar{V}_\theta}{R}-bR\partial_R\right)W_{jn} +i\partial_R\bar{W}\int_0^\infty K_{jn}(S)
W_{jn}\left(\frac{R}{S}\right)\frac{\dif S}{S^2}
=0.
\end{equation}
We remark that \eqref{eq:RSE:eigenfunction} is well defined since $W\in D(L_b)$ by construction.
On the interval $(0,r_1-\varepsilon)$, since $\bar{V}_\theta=\frac{cR}{2}$ and $\partial_R\bar{W}=0$, we deduce that
$$
W_n(R)
=C_0
R^{\frac{2(\lambda-ab)+ijn c}{2b}},
$$
for some constant $C_0$. On the interval $(r_2+\varepsilon,\infty)$,  
since $\bar{V}_\theta=\partial_R\bar{W}=0$, we deduce that
$$
W_n(R)
=C_1
R^{\frac{\lambda-ab}{b}}.
$$
In this case, since $W_n\in L^2$ and $\Re\lambda>3b$, necessarily $C_1=0$. As a consequence, we have
$\mathrm{supp}(W_n)\subset[0,r_2+\varepsilon]$. 
We have checked that $W_n$ is smooth outside the interval $B_\varepsilon([r_1,r_2])$. Let us check that it is smooth inside the interval $I_\varepsilon=B_{2\varepsilon}([r_1,r_2])$.
Notice that the last integrand in \eqref{eq:RSE:eigenfunction} is supported on $S\geq \frac{R}{r_2+\varepsilon}$.
Hence, it follows from \eqref{eq:RSE:eigenfunction} that $W_n\in H^1(I_\varepsilon)$. By bootstrapping, the same formula allows to prove that $W_n\in H^k(I_\varepsilon)$ for any $k\geq 1$.
\end{proof}

\begin{Rem}
It is straightforward to check that, by substituting the constant in front of $b$ in equation \eqref{eq:PropLb} and also adjusting Proposition \ref{prop:GrowthBound}, we can achieve $\gamma\geq k$ for any given $k\in\N$. 
Since $k=2$ is sufficient to prove Theorem \ref{thm:nonlinear}, we opt to keep it that way for the sake of simplicity.
\end{Rem}

\section{Nonlinear instability}\label{sec:nonlinear}

In this section we prove Theorem \ref{thm:nonlinear}. 
To this end, for any $k\in\N$, we consider the unique solution $\Omega_k^{\text{cor}}$ to \eqref{eq:Euler:cor} coupled with the initial condition
\begin{equation}\label{eq:Omegakcor0}
\Omega_k^{\text{cor}}|_{\tau=-k}=0.
\end{equation}
The existence and uniqueness of this solution is guaranteed by the Yudovich theory. More precisely, in Eulerian coordinates we have that
$$
\omega_{\epsilon,k}=\omega_0
+\epsilon\omega^{\text{lin}}
+\epsilon^2\omega^{\text{cor}}_k,
$$
is a solution to the Euler equation \eqref{eq:Euler} with the smooth initial condition
\begin{equation}\label{eq:omegak}
\omega_{\epsilon,k}|_{t=t_k}=(\omega_0+\epsilon\omega^{\text{lin}})(t_k)
=\frac{1}{abt_k}(\bar{\Omega}+\epsilon\Omega^{\text{lin}}(-k)),
\end{equation}
where $t_k=e^{-abk}$.

\begin{Rem}\label{rem:smoothness} 
On the one hand, recall that $\bar{\Omega}$ and $\bar{V}$ are smooth,  compactly supported, and radially symmetric. In particular, $\bar{V}\in L^2$. On the other hand, recall that $\Omega^{\text{lin}}=\Re(e^{\lambda\tau}W)$ with $W\in L_n^2\cap  C_c^2$ by Proposition \ref{prop:eigenfunction}. In particular, $V^{\text{lin}}\in L^2$ by Lemma \ref{lemma:VDVL2}.
Since the Euler equation is invariant under rotations and the initial data \eqref{eq:omegak} are $n$-fold symmetric, the uniqueness implies that necessarily $\omega_{\epsilon,k}$ (and thus also $\omega_{k}^{\text{cor}}$) are $n$-fold symmetric as well.
\end{Rem}

Thus, our task consists in showing that Theorem \ref{thm:nonlinear} holds for 
$\Omega^{\text{cor}}_k$ uniformly in $k$.
As we will see, the use of the Duhamel formula allows gaining extra exponential decay, but at the cost of having to control the bound under a stronger norm.
Following \cite{ABCDGMKpp}, we introduce the subspace $Y$ of vorticities $\Omega\in L_n^2$ satisfying
$$\Vert\Omega\Vert_Y:=\Vert\Omega\Vert_{L^2}+\Vert|X|\nabla\Omega\Vert_{L^2}+\Vert\nabla\Omega\Vert_{L^4}<\infty.$$
In \cite[Proposition 5.0.2]{ABCDGMKpp} it is proven that  $Y$ satisfies certain properties that allow closing the energy estimates. Due to the compact support of our vortices, some of them can be omitted, namely the control of the decay. In the following proposition we present a simpler version with the minimum necessary bounds, and provide a shorter proof for the sake of completeness.
\begin{prop}\label{propABCDGMKpp}
There exists $C>0$ such that
$$
\|\Omega\|_{L^\infty}
+
\left\|\frac{1}{|X|}|V|+|DV|\right\|_{L^\infty}\leq C\|\Omega\|_Y,
$$
for any $\Omega\in Y$, where $V=\nabla^\perp\Delta^{-1}\Omega$. Moreover, $V\cdot\nabla \Omega\in L_n^2$.
\end{prop}
\begin{proof}
We start by proving that 
\begin{equation}\label{eq:OmegaLinfty}
\|\Omega\|_{L^\infty}\leq C\|\Omega\|_{Y}.
\end{equation}
Given some fixed  $X\in\R^2$, let us denote by $B_X$ the unit ball centered at $X$, and by $(\Omega)_{B_X}$ the average value of $\Omega$ in $B_X$. By applying the Poincar\'e and H\"older inequalities, we deduce that
\begin{align*}
\|\Omega\|_{L^4(B_X)}
&\leq\|\Omega-(\Omega)_{B_X}\|_{L^4(B_X)}
+|B|^{\nicefrac{1}{4}}|(\Omega)_{B_X}|\\
&\leq C\|\nabla\Omega\|_{L^4(B_X)}
+|B_X|^{-\nicefrac{1}{4}}\|\Omega\|_{L^2(B_X)}
\leq C\|\Omega\|_Y,
\end{align*}
where $C$ is a constant that changes from line to line, but it is independent of $X$. Hence, by applying the Morrey inequality we get
$$
\|\Omega\|_{L^\infty(B_X)}
\leq C\|\Omega\|_{W^{1,4}(B_X)}
\leq C\|\Omega\|_Y.
$$
Since this estimate holds in every ball $B_X$, we conclude \eqref{eq:OmegaLinfty}. As a corollary, by applying the log-convexity of the $L^p$-norms, we deduce that
$$
\|\Omega\|_{L^p}
\leq C\|\Omega\|_Y,
$$
for any $2\leq p\leq \infty$. In particular, for $p=4$ this implies that $Y\hookrightarrow W^{1,4}$.
Therefore, by applying first the Morrey inequality and the Calderon-Zygmund theory, we get
$$
\|DV\|_{L^\infty}
\leq C\|DV\|_{W^{1,4}}
\leq C\|\Omega\|_{W^{1,4}}
\leq C\|\Omega\|_Y.
$$
In particular, $V$ is continuous. Therefore, letting $R\to 0$ in \eqref{eq:intV=0}, we observe that $V(0)=0$.
Consequently, by integrating $DV$ on the segment with endpoints $0$ and $X$ we get
$$|V(X)|\leq C |X|\|\Omega\|_Y.$$
For the last statement, we have
\begin{equation}\label{eq:VDOmegaL2}
\|V\cdot\nabla\Omega\|_{L^2}
\leq\left\|\frac{1}{|X|}V\right\|_{L^\infty}
\||X|\nabla\Omega\|_{L^2}\leq
\|\Omega\|_Y^2.
\end{equation}
It remains to check that $V\cdot\nabla\Omega=\mathrm{div}(V\Omega)$ is $n$-fold symmetric. By applying \eqref{eq:nfold} and \eqref{eq:Vnfold}, we get
\begin{equation}\label{eq:VOmeganfold}
\mathrm{div}(V\Omega)(X)
=\mathrm{div}(e^{-\frac{2\pi i}{n}}(V\Omega)(e^{\frac{2\pi i}{n}} X))
=\mathrm{div}(V\Omega)(e^{\frac{2\pi i}{n}} X),
\end{equation}
for any $X\in\R^2$.
\end{proof}

Since $\omega_0$, $\omega^{\text{lin}}$ and the forcing term $f$ are smooth and compactly supported for any $t\geq t_k$, it is classical that $\omega_{\epsilon,k}$ (and thus also $\omega_k^{\text{cor}}$) enjoy the same properties.
Due to \eqref{eq:Omegakcor0}, given some fixed $0<\delta_0\leq\frac12\Re\lambda$,  we can define $-k<\tau_k\leq\infty$ to be the largest time such that
\begin{equation}\label{kdependentbound}
\Vert\Omega^{\text{cor}}_k(\tau)\Vert_Y\leq e^{(\Re\lambda+\delta_0)\tau},
\end{equation}
for any $-k\leq\tau\leq\tau_k$.
Our task is to show that the times $\tau_k$ satisfying \eqref{kdependentbound} do not diverge to $-\infty$ as $k\to\infty$. If $\tau_k$ remained positive, we were done. Hence, let us assume from now on that $\tau_k\leq 0$ for all $k$ sufficiently large. 
Thus, in particular,
\begin{equation}\label{kdependentequality}
\Vert\Omega^{\text{cor}}_k(\tau_k)\Vert_Y= e^{(\Re\lambda+\delta_0)\tau_k}.
\end{equation}
Firstly, we will  improve the bound \eqref{kdependentbound} by using the following lemma. Its proof is similar to that of \cite[Lemma 5.0.4]{ABCDGMKpp} and it is the content of our Sections \ref{baselineL2estimate} and \ref{derivativeestimates}.
\begin{lemma}\label{lemma:nonlineark}
There exists $C>0$ independent of $k$ such that 
\begin{equation}\label{eq:nonlineark}
\Vert\Omega^{\text{\rm cor}}_k(\tau)\Vert_Y\leq C e^{\frac{1}{2}(3\Re\lambda+\delta_0)\tau},
\end{equation}
for all $-k\leq\tau\leq\tau_k$.
\end{lemma}

\subsection{Proof of Theorem \ref{thm:nonlinear}}
\label{sec:Proofnonlinear}
In this section we show how Lemma \ref{lemma:nonlineark} allows proving Theorem \ref{thm:nonlinear}.
Notice that \eqref{kdependentequality} together with \eqref{eq:nonlineark} implies $C e^{\frac{1}{2}(\Re\lambda-\delta_0) \tau_k}\ge 1$. Thus, we can take 
$$
\bar{\tau}
=\inf_k\tau_k\ge -\frac{2\log C}{\Re\lambda-\delta_0}>-\infty.
$$
In Eulerian coordinates, the classical solutions $\omega_{\epsilon,k}$ satisfy, for all $t\geq t_k$, the bound
\begin{align*}
\|\omega_{\epsilon,k}(t)\|_{L^p}
&\leq
\|\omega_{\epsilon,k}(t_k)\|_{L^p}
+\int_{t_k}^t\|f(s)\|_{L^p}\dif s,
\end{align*}
and similarly the corresponding velocities $v_{\epsilon,k}$ satisfy
\begin{align*}
\|v_{\epsilon,k}(t)\|_{L^2}
\leq
\|v_{\epsilon,k}(t_k)\|_{L^2}
+\int_{t_k}^t\|g(s)\|_{L^2}\dif s.
\end{align*}
By the scaling identities \eqref{eq:scalingLpw0}, \eqref{eq:scalingLpf} and Remark \ref{rem:smoothness}, the sequence of vorticities and velocities are uniformly bounded in $L_t^\infty (L^1\cap L^p)$ and $L_t^\infty L^2$, respectively. In this case, it is classical (see e.g.~\cite[Chapter 10]{MajdaBertozzi02}) that this sequence converges (by taking a subsequence if necessary) to a weak solution of the Euler equation \eqref{eq:Euler}.
Finally, by applying the sequential weak lower semicontinuity of the $L^p$-norm, the embedding $L^p\hookrightarrow L^2\cap L^\infty\hookrightarrow Y$ (recall Proposition \ref{propABCDGMKpp}), and \eqref{kdependentbound}, we get
$$
\|\Omega^{\text{cor}}(\tau)\|_{L^p}
\leq
\liminf_{k\to\infty}
\|\Omega_k^{\text{cor}}(\tau)\|_{L^p}
\leq e^{(\Re\lambda+\delta_0)\tau},
$$
for all $\tau\leq\bar{\tau}$.

\subsection{Baseline $L^2$ estimate}\label{baselineL2estimate}
In this section we deal with the $L^2$  part of the $Y$-norm  for the proof of Lemma~\ref{lemma:nonlineark}.
\begin{lemma}\label{lemma:irstinequality}
There exists $C>0$ such that
\begin{equation}\label{Lemmafirstinequality}
\Vert\Omega^{\text{\rm cor}}_k(\tau)\Vert_{L^2}\leq C e^{2\Re\lambda\tau},
\end{equation}
for all $-k\leq \tau\leq \tau_k$.
\end{lemma}
\begin{proof}
Firstly, recall that
$\Omega_k^{\text{cor}}$ satisfies equation \eqref{eq:Euler:cor}
with forcing 
$$\mathcal{F}(s)=(V^{\text{lin}}+\epsilon V^{\text{cor}})\cdot\nabla(\Omega^{\text{lin}}+\epsilon\Omega^{\text{cor}})(s)\in L_n^2,$$
by Proposition \ref{propABCDGMKpp} and Remark \ref{rem:smoothness}.
Hence, by exploiting the Duhamel formula
$$
\Omega_k^{\text{cor}}(\tau)
=-\int_{-k}^{\tau}e^{(\tau-s)L_b}\mathcal{F}(s)\dif s,
$$
together with Proposition \ref{prop:GrowthBound},
we estimate,
for some fixed $0<\delta\leq\frac{\Re\lambda}{2}$ and $C=C_\delta\geq 1$,
\begin{equation}\label{DuhamelforOmegacor}
\Vert\Omega^{\text{\rm cor}}_k(\tau)\Vert_{L^2}\leq C\int_{-k}^{\tau} e^{(\Re\lambda+\delta)(\tau-s)}\Vert\mathcal{F}(s)\Vert_{L^2}\dif s.
\end{equation}
On the one hand, $\Omega^{\text{lin}}=\Re(e^{\lambda \tau}W)$. On the other hand, $\Omega_k^{\text{cor}}$ satisfies the estimate \eqref{kdependentbound}. Thus, by applying \eqref{eq:VDOmegaL2}, we obtain the bound
$$\Vert \mathcal{F}(s)\Vert_{L^2}
\leq C\|(\Omega^{\text{lin}}+\epsilon\Omega_k^{\text{cor}})(s)\|_Y^2
\leq Ce^{{2\Re\lambda s}},$$
for all $-k\leq s\leq\tau_k$. Thus, \eqref{DuhamelforOmegacor} yields
$$\Vert\Omega^{\text{\rm cor}}_k(\tau)\Vert_{L^2}\leq C  e^{(\Re\lambda+\delta)\tau}\int_{-k}^{\tau} e^{{(\Re\lambda-\delta)s}}\dif\tau\leq C e^{{2\Re\lambda\tau}},$$
for all $-k\leq \tau\leq\tau_k$.
\end{proof}

\subsection{Energy estimates on derivatives}\label{derivativeestimates}

With \eqref{Lemmafirstinequality} proven we now proceed to show that
\begin{align}
\Vert |X|D\Omega^{\text{\rm cor}}_k(\tau_k)\Vert_{L^2}
&\leq C e^{2\Re\lambda\tau_k},\label{Lemmasecondinequality}\\
\Vert D\Omega^{\text{\rm cor}}_k(\tau_k)\Vert_{L^4}
&\leq C e^{\frac{1}{2}(3\Re\lambda+\delta_0)\tau},\label{Lemmathirdinequality}
\end{align}
for all $-k\leq \tau\leq \tau_k$.
These estimates appear in Lemmas \ref{lemma:DthetaOmega} and \ref{lemma:DrOmega}.
Recalling  \eqref{LbExpression}, we rewrite \eqref{eq:Euler:cor} as
\begin{equation}\label{nonlineareq}
\partial_\tau \Omega^{\text{cor}}_k - ab\Omega^{\text{cor}}_k+(\bar{V}-bX+\mathcal{V})\cdot\nabla\Omega^{\text{cor}}_k
+\mathcal{G}
=0,
\end{equation}
where we have abbreviated
\begin{align*}
\mathcal{V}
&:=\epsilon V^{\text{lin}}+\epsilon^2 V^{\text{cor}}_k,\\
\mathcal{G}
&:=V^{\text{lin}}\cdot\nabla\Omega^{\text{lin}}
+\epsilon  V^{\text{cor}}_k\cdot\nabla\Omega^{\text{lin}}+V^{\text{cor}}_k\cdot\nabla\bar{\Omega}.
\end{align*}
Throughout the energy estimates, we will need to control  the norms appearing in \eqref{Lemmasecondinequality}\eqref{Lemmathirdinequality} for the terms involving $\mathcal{V}$ and $\mathcal{G}$. These bounds appear in the next two lemmas.

\begin{lemma}\label{lemma:Claim:V}
There exists $C>0$ such that
\begin{align*}
\left\||X|\left(\frac{1}{|X|}|\mathcal{V}|+|D\mathcal{V}|\right)|D\Omega^{\text{cor}}_k|(\tau)\right\|_{L^2}
&\leq Ce^{(2\Re\lambda+\delta_0)\tau},\\
\left\|\left(\frac{1}{|X|}|\mathcal{V}|+|D\mathcal{V}|\right)|D\Omega^{\text{cor}}_k|(\tau)\right\|_{L^4}
&\leq Ce^{(2\Re\lambda+\delta_0)\tau},
\end{align*}
for all $-k\leq\tau\leq\tau_k$.
\end{lemma}
\begin{proof}
It follows by applying Proposition \ref{propABCDGMKpp}.
Recall that $\Omega^{\text{lin}}$ is smooth, compactly supported, and proportional to $e^{\lambda t}$, and that $\Omega_k^{\text{cor}}$ satisfies the bound \eqref{kdependentbound}.
\end{proof}

\begin{lemma}\label{lemma:Claim:G}
There exists $C>0$ such that
\begin{align}
\Vert |X|D\mathcal{G}(\tau)\Vert_{L^2}
&\leq Ce^{2\Re\lambda\tau},\label{Claimeq2}\\
\Vert D\mathcal{G}(\tau)\Vert_{L^4}
&\leq Ce^{\tfrac{1}{2}(3\Re\lambda+\delta_0)\tau},\label{Claimeq1}
\end{align}
for all $-k\leq\tau\leq\tau_k$.
\end{lemma}
\begin{proof}
We have 
\begin{align*}
\partial_j\mathcal{G
}&= \partial_jV^{\text{lin}}\cdot\nabla\Omega^{\text{lin}}
+\epsilon  \partial_jV^{\text{cor}}_k\cdot\nabla\Omega^{\text{lin}}+\partial_jV^{\text{cor}}_k\cdot\nabla\bar{\Omega}\\
&+   V^{\text{lin}}\cdot\nabla\partial_j\Omega^{\text{lin}}
+\epsilon  V^{\text{cor}}_k\cdot\nabla\partial_j\Omega^{\text{lin}}+V^{\text{cor}}_k\cdot\nabla\partial_j\bar{\Omega}.
\end{align*}
We proceed to bound the worst behaved terms. These are the terms involving both $V_k^{\text{cor}}$ and $\bar{\Omega}$.
Recall that $\bar{\Omega}$ (and $\Omega^{\text{lin}}$) are smooth and supported on a ball of radius $\bar{R}$.
On the one hand, by applying Lemma \ref{lemma:VDVL2} coupled with the baseline $L^2$ estimate (Lemma \ref{lemma:irstinequality}), we get
$$
\frac{1}{\bar{R}}\|V_k^{\text{cor}}\|_{L^2(B_{\bar{R}})}+\|DV_k^{\text{cor}}\|_{L^2}\leq C\|\Omega_k^{\text{cor}}\|_{L^2}
\leq Ce^{2\Re\lambda\tau}.
$$
Therefore, we deduce that
\begin{align*}
\||X|DV_k^{\text{cor}}D\bar{\Omega}\|_{L^2}
\leq
\bar{R}\|DV_k^{\text{cor}}\|_{L^2}
\|D\bar{\Omega}\|_{L^\infty}
\leq Ce^{2\Re\lambda\tau},\\
\||X|V_k^{\text{cor}}D^2\bar{\Omega}\|_{L^2}
\leq
\bar{R}\|V_k^{\text{cor}}\|_{L^2(B_{\bar{R}})}
\|D^2\bar{\Omega}\|_{L^\infty}
\leq Ce^{2\Re\lambda\tau},
\end{align*}
and, by additionally applying Proposition \ref{propABCDGMKpp} coupled with \eqref{kdependentbound} and the baseline estimate , we get
\begin{align*}
\|DV_k^{\text{cor}}D\bar{\Omega}\|_{L^4}
\leq\|DV_k^{\text{cor}}\|_{L^2}^{\nicefrac{1}{2}}
\|DV_k^{\text{cor}}\|_{L^\infty}^{\nicefrac{1}{2}}
\|D\bar{\Omega}\|_{L^\infty}
\leq C e^{\tfrac{1}{2}(3\Re\lambda+\delta_0)\tau},\\
\|V_k^{\text{cor}}D^2\bar{\Omega}\|_{L^4}
\leq\|V_k^{\text{cor}}\|_{L^2(B_{\bar{R}})}^{\nicefrac{1}{2}}
\|V_k^{\text{cor}}\|_{L^\infty(B_{\bar{R}})}^{\nicefrac{1}{2}}
\|D^2\bar{\Omega}\|_{L^\infty}
\leq C e^{\tfrac{1}{2}(3\Re\lambda+\delta_0)\tau}.
\end{align*}
The remaining terms are bounded similarly and have the same (or faster) decay.
\end{proof}

Notice that until now we have neglected the terms $\bar{V}$ and $X$.  In fact, these terms are very problematic since they do not give any extra decay. Following \cite{ABCDGMKpp}, we bypass this problem by estimating the derivatives with respect to the polar coordinates 
$$
D
=e_R\partial_R+e_\theta\frac{1}{R}\partial_\theta.
$$
More precisely, 
by noticing that
\begin{equation}\label{eq:VX}
\bar{V}\cdot\nabla\Omega^{\text{cor}}_k=\frac{\bar{V}_\theta(R)}{R}\partial_\theta\Omega^{\text{cor}}_k,
\quad\quad
X\cdot\nabla\Omega^{\text{cor}}_k=R\partial_R\Omega^{\text{cor}}_k,
\end{equation}
it becomes convenient to perform first energy estimates with respect to the angular derivative. This allows to gain an extra decay on $\partial_\theta\Omega^{\text{cor}}_k$, and then close the energy estimates for  $\partial_R\Omega^{\text{cor}}_k$.

\subsubsection{Estimates on the angular derivative}

\begin{lemma}\label{lemma:DthetaOmega} 
There exists $C>0$ such that
\begin{align}
\Vert \partial_\theta\Omega^{\text{cor}}_k(\tau)\Vert_{L^2}
&\leq Ce^{2\Re\lambda\tau},\label{AngularDerivativeBound1}\\
\left\Vert \frac{1}{R}\partial_\theta\Omega^{\text{cor}}_k(\tau)\right\Vert_{L^4}
&\leq Ce^{\tfrac{1}{2}(3\Re\lambda + \delta_0)\tau},\label{AngularDerivativeBound2}
\end{align}
for all $-k\leq\tau\leq\tau_k$.
\end{lemma}
\begin{proof}
For \eqref{AngularDerivativeBound1}, by differentiating \eqref{nonlineareq} with respect to $\theta$, we get
\begin{equation}\label{diffthetaeq}
\begin{split}
\partial_\tau\partial_\theta \Omega^{\text{cor}}_k - ab\partial_\theta\Omega^{\text{cor}}_k+(\bar{V}-bX+\mathcal{V})\cdot\nabla\partial_\theta\Omega^{\text{cor}}_k&\\
=
-\partial_\theta\mathcal{V}_R\partial_R\Omega_k^{\text{cor}}
-\partial_\theta\mathcal{V}_\theta\frac{\partial_\theta\Omega_k^{\text{cor}}}{R}
-\partial_\theta\mathcal{G}&=:\mathcal{H}_1,
\end{split}
\end{equation}
where we have applied \eqref{eq:VX}
$$
\partial_\theta(\bar{V}\cdot\nabla\Omega^{\text{cor}}_k)
=\bar{V}\cdot\nabla\partial_\theta\Omega^{\text{cor}}_k,
\quad\quad
\partial_\theta(X\cdot\nabla\Omega^{\text{cor}}_k)
=X\cdot\nabla\partial_\theta\Omega^{\text{cor}}_k,
$$
and also
$$
\partial_\theta(\mathcal{V}\cdot\nabla\Omega_k^{\text{cor}})
=\mathcal{V}\cdot\nabla\partial_\theta\Omega_k^{\text{cor}}
+\partial_\theta\mathcal{V}_R\partial_R\Omega_k^{\text{cor}}
+\partial_\theta\mathcal{V}_\theta\frac{\partial_\theta\Omega_k^{\text{cor}}}{R}.
$$
Multiplying \eqref{diffthetaeq} by $\partial_\theta\Omega^{\text{cor}}_k$ and integrating by parts yields
$$
\frac{1}{2}\partial_\tau\Vert \partial_\theta\Omega^{\text{cor}}_k\Vert_{L^2}^2=b(a-1)\Vert \partial_\theta\Omega^{\text{cor}}_k\Vert_{L^2}^2+\int_{\R^2}\mathcal{H}_1\partial_\theta\Omega^{\text{cor}}_k\dif X
\leq
\Vert \mathcal{H}_1\Vert_{L^2} \Vert \partial_\theta\Omega^{\text{cor}}_k\Vert_{L^2},
$$
where we have applied that $\mathrm{div}X=2$ and $a<1$. Thus,
$$
\Vert\partial_\theta\Omega^{\text{cor}}_k(\tau)\Vert_{L^2}
\leq\int_{-k}^{\tau}
\Vert\mathcal{H}_1(s)\Vert_{L^2}\dif s.
$$
Then, \eqref{AngularDerivativeBound1} follows by applying Lemmas \ref{lemma:Claim:V} and \ref{lemma:Claim:G} since
$$
\|\mathcal{H}_1(s)\|_{L^2}
\leq C e^{2\Re\lambda s},
$$
for all $-k\leq s\leq\tau_k$.
For \eqref{AngularDerivativeBound2}, multiplying \eqref{diffthetaeq} by $\frac{1}{R}$ yields
\begin{equation}\label{diffthetaeqr}
\begin{split}
\partial_\tau \frac{\partial_\theta\Omega^{\text{cor}}_k}{R} - b(a+ 1)\frac{\partial_\theta\Omega^{\text{cor}}_k}{R}+(\bar{V}-bX+\mathcal{V})\cdot\nabla\frac{\partial_\theta\Omega^{\text{cor}}_k}{R}&\\
=
\frac{1}{R}\mathcal{V}_R\frac{\partial_\theta\Omega^{\text{cor}}_k}{R}
+\frac{1}{R}\mathcal{H}_1&=:\mathcal{H}_2,
\end{split}
\end{equation}
where $\mathcal{V}_R=\mathcal{V}\cdot e_R$, and we have applied that $\bar{V}=\bar{V}_\theta e_\theta$, $X=Re_R$ and
$
\nabla\left(\frac{1}{R}\right)
=-\frac{1}{R^2}e_R.
$
After multiplying \eqref{diffthetaeqr} by $\left(\frac{1}{R}\partial_\theta\Omega^{\text{cor}}_k\right)^3$ and integrating by parts, we get
\begin{align*}
\frac{1}{4}\partial_\tau\left\Vert \frac{1}{R}\partial_\theta\Omega^{\text{cor}}_k\right\Vert_{L^4}^4
&=b\left(a+\frac{1}{2}\right)\left\Vert \frac{1}{R}\partial_\theta\Omega^{\text{cor}}_k\right\Vert_{L^4}^4+\int_{\R^2}\mathcal{H}_2\left(\frac{1}{R}\partial_\theta\Omega^{\text{cor}}_k\right)^3\dif X\\
&\leq b\left(a+\frac{1}{2}\right)\left\Vert \frac{1}{R}\partial_\theta\Omega^{\text{cor}}_k\right\Vert_{L^4}^4+\Vert\mathcal{H}_2\Vert_{L^4}\left\Vert \frac{1}{R}\partial_\theta\Omega^{\text{cor}}_k\right\Vert_{L^4}^3.
\end{align*}
Therefore,
\begin{equation}\label{difftheta1reqestimateobjective}
\left\Vert \frac{1}{R}\partial_\theta\Omega^{\text{cor}}_k(\tau)\right\Vert_{L^4}\leq\int_{-k}^{\tau}e^{b\left(a+\frac{1}{2}\right)(\tau-s)}\Vert \mathcal{H}_2(s)\Vert_{L^4}\dif s.
\end{equation}
Then \eqref{AngularDerivativeBound2} follows by applying Lemmas \ref{lemma:Claim:V} and \ref{lemma:Claim:G} since
$$
\|\mathcal{H}_2(s)\|_{L^4}
\leq Ce^{\tfrac{1}{2}(3\Re\lambda+\delta_0)s},
$$
for all $-k\leq s\leq\tau_k$.
Notice that we need 
$
b\left(a+\frac{1}{2}\right)\ll\frac{1}{2}(3\Re\lambda+\delta_0),
$
but this is satisfied since Proposition \ref{prop:GrowthBound} guarantees $\Re\lambda>3b$.
\end{proof}

\subsubsection{Estimates on the radial derivative}
\begin{lemma}\label{lemma:DrOmega} 
There exists $C>0$ such that
\begin{align}
\Vert R\partial_R\Omega^{\text{cor}}_k(\tau)\Vert_{L^2}
&\leq Ce^{2\Re\lambda\tau},\label{RadialDerivativeBound1}\\
\left\Vert \partial_R\Omega^{\text{cor}}_k(\tau)\right\Vert_{L^4}
&\leq Ce^{\tfrac{1}{2}(3\Re\lambda+\delta_0)\tau},\label{RadialDerivativeBound2}
\end{align}
for all  $-k\leq\tau\leq\tau_k$.
\end{lemma}
\begin{proof}
For \eqref{RadialDerivativeBound1}, by differentiating \eqref{nonlineareq} with respect to $R$, we get
\begin{equation}\label{diffreq}
\begin{split}
\partial_\tau\partial_R \Omega^{\text{cor}}_k - b(a+1)\partial_R\Omega^{\text{cor}}_k+(\bar{V}-bX+\mathcal{V})\cdot\nabla\partial_R\Omega^{\text{cor}}_k&\\
=
-\partial_R\mathcal{V}_R\partial_R\Omega_k^{\text{cor}}
-\left(\partial_R\mathcal{V}_\theta-\frac{\mathcal{V}_\theta}{R}\right)\frac{\partial_\theta\Omega^{\text{cor}}_k}{R}
-\left(\partial_R\bar{V}_\theta-\frac{\bar{V}_\theta}{R}\right)\frac{\partial_\theta\Omega^{\text{cor}}_k}{R}
-\partial_R\mathcal{G}
&=:\mathcal{H}_3.
\end{split}
\end{equation}
where we have applied \eqref{eq:VX}
\begin{align*}
\partial_R(\bar{V}\cdot\nabla\Omega^{\text{cor}}_k)
&=\bar{V}\cdot\nabla\partial_R\Omega^{\text{cor}}_k
+\left(\partial_R\bar{V}_\theta-\frac{\bar{V}_\theta}{R}\right)\frac{\partial_\theta\Omega^{\text{cor}}_k}{R}
,\\
\partial_R(X\cdot\nabla\Omega^{\text{cor}}_k)
&=X\cdot\nabla\partial_R\Omega^{\text{cor}}_k
+\partial_R\Omega^{\text{cor}}_k,
\end{align*}
and also
\begin{align*}
\partial_R(
\mathcal{V}\cdot\nabla\Omega_k^{\text{cor}})
&=\mathcal{V}\cdot\nabla\partial_\theta\Omega_k^{\text{cor}}
+\partial_R\mathcal{V}_R\partial_R\Omega_k^{\text{cor}}
+\left(\partial_R\mathcal{V}_\theta-\frac{\mathcal{V}_\theta}{R}\right)\frac{\partial_\theta\Omega^{\text{cor}}_k}{R}.
\end{align*}
After multiplying \eqref{diffreq} by $R$ we get
\begin{equation}\label{diffreqrewritten}
\begin{split}
\partial_\tau (R\partial_R \Omega^{\text{cor}}_k) - a b R\partial_R\Omega^{\text{cor}}_k+(\bar{V}-bX+\mathcal{V})\cdot\nabla(R\partial_R\Omega^{\text{cor}}_k)&\\
=\mathcal{V}_R\partial_R \Omega^{\text{cor}}_k
+R\mathcal{H}_3
&=:\mathcal{H}_4,
\end{split}
\end{equation}
where we have applied that $\bar{V}_R=0$, $X=Re^{i\theta}$ and
$\nabla R=e_R$.
Multiplying \eqref{diffreqrewritten} by $R\partial_R\Omega^{\text{cor}}_k$ and integrating by parts yields
$$
\frac{1}{2}\partial_\tau\Vert R\partial_R\Omega^{\text{cor}}_k\Vert_{L^2}^2=b(a-1)\Vert \partial_R\Omega^{\text{cor}}_k\Vert_{L^2}^2+\int_{\R^2}\mathcal{H}_4 R\partial_R\Omega^{\text{cor}}_k\dif X
\leq
\Vert \mathcal{H}_4\Vert_{L^2} \Vert R\partial_R\Omega^{\text{cor}}_k\Vert_{L^2},
$$
and thus,
$$
\Vert R\partial_R\Omega^{\text{cor}}_k(\tau)\Vert_{L^2}
\leq\int_{-k}^{\tau}\|\mathcal{H}_4(s)\|_{L^2}\dif s.
$$
For \eqref{RadialDerivativeBound2}, by multiplying \eqref{diffreq} by $(\partial_R \Omega^{\text{cor}}_k)^3$ and integrating by parts we get
\begin{align*}
\frac{1}{4}\partial_\tau\|\partial_R\Omega^{\text{cor}}_k\|_{L^4}^4
&= b\left(a+\frac{1}{2}\right)\|\partial_R\Omega^{\text{cor}}_k\|_{L^4}^4
+\int_{\R^2}\mathcal{H}_3
(\partial_R\Omega^{\text{cor}}_k)^3\dif X\\
&\leq
 b\left(a+\frac{1}{2}\right)\|\partial_R\Omega^{\text{cor}}_k\|_{L^4}^4
+\|\mathcal{H}_3\|_{L^4}\|\partial_R\Omega^{\text{cor}}_k\|_{L^4}^3,
\end{align*}
and thus
$$
\Vert \partial_R\Omega^{\text{cor}}_k(\tau)\Vert_{L^4}
\leq\int_{-k}^{\tau}e^{b\left(a+\frac{1}{2}\right)(\tau-s)}\|\mathcal{H}_3(s)\|_{L^4}\dif s.
$$
Similarly to Lemma \ref{lemma:DthetaOmega}, we conclude the proof by applying Lemmas \ref{lemma:Claim:V}  and \ref{lemma:Claim:G}, but now coupled with Lemma \ref{lemma:DthetaOmega}.
This was the reason for estimating the angular derivative before.
\end{proof}


\appendix

\section{Proof of formula \eqref{eq:Kn}}\label{sec:residue}

We start by writing $(z=e^{i\vartheta})$
$$
K_n(\rho)
=\frac{1}{2\pi}\int_0^{2\pi}
\frac{\sin(\vartheta)\sin(n\vartheta)}{1+\rho^2-2\rho\cos(\vartheta)}\dif\vartheta
=\frac{1}{8\pi i \rho}\int_{\T}
f(z)
\dif z,
$$
where $f$ is the meromorphic function
$$
f(z)
=\frac{(z^2-1)(z^{2n}-1)}{z^{n+1}(z-\rho)(z-\rho^{-1})}.
$$
Hence, by the Residue Theorem we have
$$
K_n(\rho)
=\frac{1}{4\rho}
\sum_{z=0,\rho,\rho^{-1}}
\mathrm{Ind}(\T,z)
\mathrm{Res}(f,z).
$$
We have
$$
\mathrm{Ind}(\T,0)=1,
\quad\quad
\mathrm{Ind}(\T,\rho)=1_{(0,1)}(\rho).
$$
It is easy to compute
the residue for the simple poles
\begin{align*}
\mathrm{Res}(f,\rho)
&=\frac{(\rho^2-1)(\rho^{2n}-1)}{\rho^{n+1}(\rho-\rho^{-1})}
=\rho^n-\rho^{-n},\\
\mathrm{Res}(f,\rho^{-1})
&=\frac{(\rho^{-2}-1)(\rho^{-2n}-1)}{\rho^{-(n+1)}(\rho^{-1}-\rho)}
=\rho^{-n}-\rho^n.
\end{align*}
For the multiple pole $z=0$ we write the Maclaurin series of $f$
$$
f(z)
=\frac{(z^2-1)(z^{2n}-1)}{z^{n+1}}g(z)
=\left(z^{n+1}-z^{n-1}-\frac{1}{z^{n-1}}+\frac{1}{z^{n+1}}\right)g(z),
$$
where
$$
g(z)=\frac{1}{(z-\rho)(z-\rho^{-1})}
=\frac{1}{\rho-\rho^{-1}}
\left(
\frac{1}{z-\rho}-\frac{1}{z-\rho^{-1}}
\right)
=\frac{1}{\rho-\rho^{-1}}
\sum_{j=0}^{\infty}
(\rho^{j+1}-\rho^{-(j+1)})z^j.
$$
Hence, for any $n\geq 1$, we deduce that
$$
\mathrm{Res}(f,0)
=\rho^n+\rho^{-n}.
$$
By plugging everything together we get \eqref{eq:Kn}.

\subsection*{Acknowledgments}
The authors thank all participants of the reading seminar on the uniqueness or lack thereof for the Navier-Stokes and Euler equations, which took place in Madrid during the spring of 2023.
D.F.~and F.M.~acknowledge the hospitality and financial support of the Institute for Advanced Study in Princeton during various periods in 2021-2022, as well as the discussions held there with Camillo De Lellis and his group on topics related to Vishik's construction. 

A.C.~ acknowledges financial support from Grant PID2020-114703GB-I00 funded by MCIN /AEI/ 10.13039/501100011033, Grants RED2022-
134784-T and RED2018-102650-T funded by MCIN/ AEI/ 10.13039/501100011033 and by a 2023
Leonardo Grant for Researchers and Cultural Creators, BBVA Foundation. The BBVA Foundation
accepts no responsibility for the opinions, statements, and contents included in the project and/or
the results thereof, which are entirely the responsibility of the authors.
A.C.~and D.F.~ acknowledge financial support from the Severo Ochoa Programme
for Centres of Excellence Grant CEX2019-000904-S funded by MCIN/ AEI/ 10.13039/501100011033. D.F.~acknowledges the financial support of QUAMAP and the ERC Advanced Grant 834728. D.F., F.M.~and M.S.~acknowledge support from PI2021-124-195NB-C32. D.F.~and F.M.~acknowledge support from CM through the Line of Excellence for University Teaching Staff between CM and UAM.
F.M.~acknowledges support from the 
Max Planck Institute for Mathematics in the Sciences. 
M.S.~acknowledges support from PID2022-136589NB-I00 founded by MCIN/AEI/10.13039/501100011033/ and by FEDER ({\it Una manera de hacer Europa}) and from ``Conselleria d'Innovaci\'o, Universitats, Ci\`encia i Societat Digital'': project AICO/2021/223 and programme ``Subvenciones para la contratación de personal investigador en fase postdoctoral'' (APOSTD 2022),  Ref. CIAPOS/2021/28.

\bibliographystyle{abbrv}
\bibliography{Unstable_Euler_vortex}

\begin{flushleft}
	\bigskip
	\'Angel Castro\\
	\textsc{Instituto de Ciencias Matemáticas, CSIC-UAM-UC3M-UCM\\
		28049 Madrid, Spain}\\
	\textit{E-mail address:} angel\_castro@icmat.es 
\end{flushleft}

\begin{flushleft}
	\bigskip
	Daniel Faraco\\
	\textsc{Departamento de Matem\'aticas, Universidad Aut\'onoma de Madrid, Instituto de
Ciencias Matem\'aticas, CSIC-UAM-UC3M-UCM\\
		28049 Madrid, Spain}\\
	\textit{E-mail address:} 
daniel.faraco@uam.es 
\end{flushleft}

\begin{flushleft}
	\bigskip
	Francisco Mengual\\
	\textsc{Max Planck Institute for Mathematics in the Sciences\\
		04103 Leipzig, Germany}\\
	\textit{E-mail address:} fmengual@mis.mpg.de
\end{flushleft}

\begin{flushleft}
	\bigskip
	Marcos Solera\\
	\textsc{Departament d’An\`alisi Matem\`atica, Universitat de Val\`encia\\
		46100 Burjassot, Spain}\\
	\textit{E-mail address:} marcos.solera@uv.es
\end{flushleft}

\end{document}